\documentclass[a4paper,12pt]{article}
\usepackage[top=2cm, bottom=2cm, outer=2.5cm, inner=2cm, headsep=14pt]{geometry}
\usepackage[dvips]{epsfig}
\usepackage{graphicx}
\usepackage{amsmath,amssymb,euscript,graphicx,amsfonts,verbatim}
\usepackage{color}
\usepackage{enumitem}
\usepackage{tikz}

\definecolor{ForestGreen}{RGB}{12, 110, 46}
\definecolor{ForestGreenTwo}{RGB}{120, 110, 86}

\newtheorem{theorem}{Theorem}[section]
\newtheorem{proposition}[theorem]{Proposition}
\newtheorem{lemma}[theorem]{Lemma}
\newtheorem{corollary}[theorem]{Corollary}

\newtheorem{remark}[theorem]{Remark}

\newcommand{\G}{\Gamma}

\newcommand{\Qed}{\rule{2.5mm}{3mm}}
\newenvironment{proof}{{\noindent \sc Proof.}}{\hfill $\Qed$}
\def\ZZ{{\hbox{\sf Z\kern-.43emZ}}}

\title{ON CERTAIN REGULAR NICELY DISTANCE--BALANCED GRAPHS}

\author{
{Blas Fern\'andez}\\
{\small UP IAM}\\
{\small University of Primorska}\\
{\small Muzejski trg 2, 6000 Koper, Slovenia }\\
{\small blas.fernandez@famnit.upr.si} \and
{\v Stefko Miklavi\v c}\\
{\small UP IAM and UP FAMNIT }\\
{\small University of Primorska}\\
{\small Muzejski trg 2, 6000 Koper, Slovenia }\\
{\small stefko.miklavic@upr.si} \and
{Safet Penji\'c}\\
{\small UP IAM and UP FAMNIT }\\
{\small University of Primorska}\\
{\small Muzejski trg 2, 6000 Koper, Slovenia }\\
{\small safet.penjic@upr.si}
}

\begin{document}
{\small
\maketitle
}

\begin{abstract}
A connected graph $\G$ is called {\em nicely distance--balanced}, whenever there exists a positive integer $\gamma=\gamma(\G)$, such that for any two adjacent vertices $u,v$ of $\G$ there are exactly $\gamma$ vertices of $\G$ which are closer to $u$ than to $v$, and exactly $\gamma$ vertices of $\G$ which are closer to $v$ than to $u$. Let $d$ denote the diameter of $\G$. It is known that $d \le \gamma$, and that nicely distance-balanced graphs with $\gamma = d$ are precisely complete graphs and cycles of length $2d$ or $2d+1$. In this paper we classify regular nicely distance-balanced graphs with $\gamma=d+1$.   
\end{abstract}




\section{Introduction}
\label{sec:intro}

\noindent
Let $\G$ be a finite, undirected, connected graph with diameter $d$, and let  $V(\G)$ and $E(\G)$ denote the vertex set and the edge set of $\G$, respectively.
For $u,v \in V(\G)$, let $\G(u)$ be the set of neighbours of $u$, and let $d(u,v) = d_{\G}(u,v)$ denote the minimal path-length distance between $u$ and $v$. 
For a pair of adjacent vertices $u,v$ of $\G$ we denote 
$$
W_{u,v} = \{x\in V(\G)\mid d(x,u)<d(x,v)\}.
$$
We say that $\G$ is {\em distance--balanced} (DB for short) whenever
for an arbitrary pair of adjacent vertices $u$ and $v$ of $\G$ we have that
$$
|W_{u,v}| = |W_{v,u}|.
$$

The investigation of distance-balanced graphs was initiated in 1999 by Handa~\cite{Ha}, although the name {\em distance-balanced} was coined nine years later by Jerebic, Klav\v zar and Rall~\cite{JKR}. The family of distance-balanced graphs is very rich and its study is interesting from various purely graph-theoretic aspects where one focuses on particular properties of such graphs such as symmetry~\cite{KMMM05}, connectivity~\cite{Ha, MS} or complexity aspects of algorithms related to such graphs~\cite{CL}. However, the balancedness property of these graphs makes them very appealing also in areas such as mathematical chemistry and communication networks. For instance, the investigation of such graphs is highly related to the well-studied Wiener index and Szeged index (see ~\cite{BBCKVZ14, IKM, JKR}) and they present very desirable models in various real-life situations related to (communication) networks~\cite{BBCKVZ14}. It turns out that these graphs can be characterized by properties that at first glance do not seem to have much in common with the original definition from~\cite{JKR}. For example, in~\cite{BCPSSS} it was shown that the distance-balanced graphs coincide with the {\em self-median} graphs, that is graphs for which the sum of the distances from a given vertex to all other vertices is independent of the chosen vertex. Other such examples are  {\em equal opportunity graphs} (see~\cite{BBCKVZ14} for the definition). In~\cite{BBCKVZ14} it is shown that even order distance-balanced graphs are also equal to opportunity graphs.

\medskip
The notion of nicely distance-balanced graphs appears quite naturally in the context of DB graphs. We say that $\G$ is {\em nicely distance--balanced} (NDB for short) whenever there exists a positive integer $\gamma=\gamma(\G)$,
such that for an arbitrary pair of adjacent vertices $u$ and $v$ of $\G$
$$
  |W_{u,v}| = |W_{v,u}| = \gamma
$$
holds. Clearly, every NDB graph is also DB, but the opposite is not necessarily true. For example, if $n \ge 3$ is an odd positive integer, then the prism graph on $2n$ vertices is DB, but not NDB. 

Assume now that $\G$ is NDB. Let us denote the diameter of $\G$ by $d$.  In \cite{KM}, where these graphs were first defined, it was proved that $d \le \gamma$ and NDB graphs with $d=\gamma$ were classified. It turns out that $\G$ is NDB with $d=\gamma$  if and only if $\G$ is either isomorphic to a complete graph on $n \ge 2$ vertices,  or to a cycle on $2d$ or $2d+1$ vertices. In this paper we study NDB graphs for which $\gamma=d+1$. The situation in this case is much more complex than in the case $\gamma=d$. Therefore, we will concentrate our study to the class of regular graphs (recall that $\G$ is said to be regular with valency $k$  if $|\G(u)|=k$ for every $u \in V(\G)$).  Our main result is the following theorem.

\begin{theorem}
	\label{thm:main}
	Let $\G$ be a regular NDB graph with valency $k$ and diameter $d$. Then $\gamma = d+1$ if and only if $\G$ is isomorphic to one of the following graphs:
\begin{enumerate}
\item the Petersen graph (with $k=3$ and $d=2$); 
\item the complement of the Petersen graph (with $k=6$ and $d=2$); 
\item the complete multipartite graph $K_{t \times 3}$ with $t$ parts of cardinality $3$, $t \ge 2$ (with $k=3(t-1)$ and $d=2$); 
\item the M\"obius ladder graph on eight vertices (with $k=3$ and $d=2$); 
\item the Paley graph on 9 vertices (with $k=4$ and $d=2$); 
\item the $3$-dimensional hypercube $Q_3$ (with $k=3$ and $d=3$);
\item the line graph of the $3$-dimensional hypercube $Q_3$ (with $k=4$ and $d=3$); 
\item the icosahedron (with $k=5$ and $d=3$). 
\end{enumerate}
\end{theorem}

Our paper is organized as follows. After some preliminaries in Section \ref{sec:prelim} we prove certain structural results about NDB graphs with $\gamma=d+1$ in Section \ref{sec:struc}. In Section \ref{sec:regular} we show that if $\G$ is a regular NDB graph with $\gamma=d+1$, then $d \le 5$ and  the valency of $\G$ is either $3$, $4$ or $5$. In Sections \ref{sec:k=3}, \ref{sec:k=4} and \ref{sec:k=5} we consider each of these three cases separately. 


\section{Preliminaries}
\label{sec:prelim}

\noindent
In this section we recall some preliminary results that we will find useful later in the paper. Let $\G$ denote a simple, finite, connected graph with vertex set $V(\G)$, edge set $E(\G)$. If $u,v \in V(\G)$ are adjacent then we simply write $u \sim v$ and we denote the corresponding edge by $uv=vu$. For $u\in V(\G)$ and an integer $i$ we let $\G_{i}(u)$ denote the set of vertices of $V(\G)$ that are at distance $i$ from $u$. We abbreviate $\G(u)=\G_1(u)$. We set $\epsilon(u)=\max\{d(u,z) \mid z \in V(\G)\}$ and we call $\epsilon(u)$ the {\em eccentricity} of  $u$. Let  $d = \max \{\epsilon(u) \mid u \in V(\G)\}$ denote the {\em diameter} of $\G$. Pick adjacent vertices $u,v$ of $\G$. For any two non-negative integers $i,j$ we let 
$$
D^i_j(u,v)=\G_i(u)\cap\G_j(v).
$$
By the triangle inequality we observe only the sets $D^{i-1}_i(u,v)$, $D^{i}_i(u,v)$ and $D^{i}_{i-1}(u,v)$ ($1\le i \le d$) can be nonempty. Moreover, the next result holds.

\begin{lemma}
	\label{pr} 
	With the above notation, abbreviate $D^i_j=D^i_j(u,v)$. Then  the following {\rm (i)-(iv)} hold for $1 \le i \le d$.  
	\begin{enumerate}[label={\rm(\roman*)}]
		\item  If $w\in D^{i}_{i-1}$ then $\G(w) \subseteq  D^{i-1}_{i-2} \cup  D^{i-1}_{i-1} \cup  D^{i-1}_{i} \cup  D^{i}_{i-1} \cup  D^{i}_{i} \cup  D^{i+1}_{i}$. 
		\item If $w\in D^{i}_{i}$  then $\G(w) \subseteq   D^{i-1}_{i-1} \cup  D^{i-1}_{i} \cup  D^{i}_{i-1} \cup  D^{i}_{i} \cup  D^{i}_{i+1} \cup  D^{i+1}_{i} \cup  D^{i+1}_{i+1}$.  
		\item  If $w\in D^{i-1}_{i}$ then $\G(w) \subseteq   D^{i-2}_{i-1} \cup  D^{i-1}_{i-1} \cup  D^{i-1}_{i} \cup  D^{i}_{i-1} \cup  D^{i}_{i} \cup  D^{i}_{i+1}$.   
		\item If $D^{i}_{i+1} \ne \emptyset$ ($D^{i+1}_i \ne \emptyset$, respectively) then  $D^{j}_{j+1} \neq \emptyset$ ($D^{j+1}_{j} \neq \emptyset$, respectively) for every $0 \leq j \leq i$. 
	\end{enumerate}   
\end{lemma}
\begin{proof}
	Straightforward (see also Figure~\ref{01}).
\end{proof}

{\small\begin{figure}[!ht]{\rm
\begin{center}
\begin{tikzpicture}[scale=.565]

\draw [line width=1pt, draw=ForestGreen] (-2.,-3.)-- (22.,-3.);
\draw [line width=1pt, draw=ForestGreen] (-2.,3.)-- (22.,3.);
\draw [line width=1pt, draw=ForestGreen] (1.,0.)-- (25,0);
\draw [line width=1pt, draw=ForestGreen] (-2,-3)-- (4,3);
\draw [line width=1pt, draw=ForestGreen] (4,-3)-- (10,3);
\draw [line width=1pt, draw=ForestGreen] (10,-3)-- (16,3);
\draw [line width=1pt, draw=ForestGreen] (16,-3)-- (22,3);
\draw [line width=1pt, draw=ForestGreen] (22,-3)-- (25,0);
\draw [line width=1pt, draw=ForestGreen] (-2.,3.)-- (4,-3);
\draw [line width=1pt, draw=ForestGreen] (4,3.)-- (10,-3);
\draw [line width=1pt, draw=ForestGreen] (10,3.)-- (16,-3);
\draw [line width=1pt, draw=ForestGreen] (16,3.)-- (22,-3);
\draw [line width=1pt, draw=ForestGreen] (22,3.)-- (25,0);
\draw [line width=1pt, draw=ForestGreen] (-2,-3)-- (-2,3);
\draw [line width=1pt, draw=ForestGreen] (4,-3)-- (4,3);
\draw [line width=1pt, draw=ForestGreen] (10,-3)-- (10,3);
\draw [line width=1pt, draw=ForestGreen] (16,-3)-- (16,3);
\draw [line width=1pt, draw=ForestGreen] (22,-3)-- (22,3);

\fill (-2.,-3.) circle [radius=0.23];
\fill (-2.,3.) circle [radius=0.23];
\node at (-2.5,3.1) {\normalsize $u$};
\node at (-2.5,-3.1) {\normalsize $v$};

\draw[fill=white, draw=white, line width=0.6pt] (4.,3.) ellipse (1.5cm and .9cm);
\draw[fill=white, draw=white, line width=0.6pt] (7.,0.) ellipse (1.5cm and .9cm);
\draw[fill=white, draw=black, line width=0.6pt] (1.,0.) ellipse (1.5cm and .9cm);
\draw[fill=white, draw=white, line width=0.6pt] (4.,-3.) ellipse (1.5cm and .9cm);
\draw[fill=white, draw=white, line width=0.6pt] (19.,0.) ellipse (1.5cm and .9cm);
\draw[fill=white, draw=black, line width=0.6pt] (22.,3.) ellipse (1.5cm and .9cm);
\draw[fill=white, draw=black, line width=0.6pt] (22.,-3.) ellipse (1.5cm and .9cm);
\draw[fill=white, draw=black, line width=0.6pt] (10.,3.) ellipse (1.5cm and .9cm);
\draw[fill=white, draw=black, line width=0.6pt] (10.,-3.) ellipse (1.5cm and .9cm);
\draw[fill=white, draw=black, line width=0.6pt] (13.,0.) ellipse (1.5cm and .9cm);
\draw[fill=white, draw=white, line width=0.6pt] (16.,3.) ellipse (1.5cm and .9cm);
\draw[fill=white, draw=white, line width=0.6pt] (16.,-3.) ellipse (1.5cm and .9cm);
\draw[fill=white, draw=black, line width=0.6pt] (25.,0.) ellipse (1.5cm and .9cm);

\node at (1,0) {\normalsize $D^{1}_{1}$};
\node at (7,0) {\normalsize $\cdots$};
\node at (13,0) {\normalsize $D^{i}_{i}$};
\node at (19,0) {\normalsize  $\cdots$};
\node at (25,0) {\normalsize $D^{d}_{d}$};

\node at (4,3) {\normalsize $\cdots$};
\node at (10,3) {\normalsize $D^{i-1}_{i}$};
\node at (16,3) {\normalsize  $\cdots$};
\node at (22,3) {\normalsize $D^{d-1}_{d}$};

\node at (4,-3) {\normalsize $\cdots$};
\node at (10,-3) {\normalsize $D^{i}_{i-1}$};
\node at (16,-3) {\normalsize  $\cdots$};
\node at (22,-3) {\normalsize $D^{d}_{d-1}$};
\end{tikzpicture}
\caption{\rm 
Graphical representation of the sets $D^i_j(u,v)$. The line between $D_j^i$ and $D_m^n$ indicates possible edges between vertices of $D^i_j$ and $D_m^n$.
}
\label{01}
\end{center}
}\end{figure}
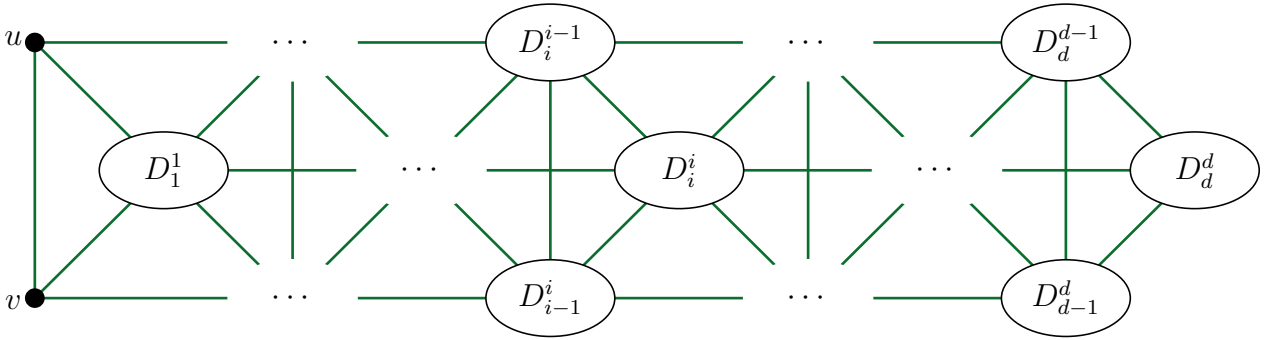}

\noindent
Let us recall the definition of the NDB graphs. For an edge $uv$ of $\G$ we denote 
$$
W_{u,v} = \{x\in V(\G)\mid d(x,u)<d(x,v)\}.
$$
We say that $\G$ is {\em nicely distance--balanced} (NDB for short) whenever there exists a positive integer $\gamma=\gamma(\G)$, such that for any edge $uv$ of $\G$
$$
|W_{u,v}| = |W_{v,u}| = \gamma
$$
holds. One can easily see that $\G$ is NDB if and only if for every edge $uv\in E(\G)$ we have

\begin{equation}
	\label{er}
	\sum_{i=1}^{d}|D^{i}_{i-1}(u,v)|=\sum_{i=1}^{d}|D^{i-1}_i(u,v)|=\gamma.
\end{equation}
Pick adjacent vertices $u,v$ of $\G$. For the purposes of this paper we say that the edge $uv$ is {\em balanced}, if \eqref{er} holds for vertices $u,v$ with $\gamma=d+1$.

Graph $\G$ is said to be {\em regular}, if there exists a non-negative integer $k$, such that $|\G(u)|=k$ for every vertex $u \in V(\G)$. In this case we also say that $\G$ is regular with {\em valency} $k$ (or $k$-regular for short). The following simple observation about regular graphs will be very useful in the rest of the paper.
\begin{lemma}
\label{eq}
Let $\G$ denote a connected regular graph. Then for every edge $uv$ of $\G$ we have 
$$
|D^1_2(u,v)| = |D^2_1(u,v)|.
$$
\end{lemma}
\begin{proof}
	Note that $\G(u)=\{v\} \cup D^1_1(u,v) \cup D^1_2(u,v)$ and $\G(v)=\{u\} \cup D^1_1(u,v) \cup D^2_1(u,v)$. As $\G$ is regular, the claim follows.
\end{proof}

\bigskip 

Assume $\G$ is regular with valency $k$.  If there exists a non-negative integer $\lambda$, such that  every pair $u,v$ of adjacent vertices of $\G$ has exactly $\lambda$ common neighbours (that is, if $|D^1_1(u,v)|=\lambda$), then we say that $\G$ is {\em edge-regular} (with parameter $\lambda$).  Before we start with our study of regular NDB graphs with $\gamma=d+1$ we have a remark.
 
\begin{remark}
	\label{rem:d=2}
	Let $\G$ denote a regular NDB graph with diameter $d$ and $\gamma=d+1$. Observe first that $d\geq 2$. Moreover, if $d=2$ then it follows from {\rm \cite[Theorem~5.2]{KM}} that $\G$ is one of the following graphs:
	\begin{enumerate}
		\item the Petersen graph,
		\item the complement of the Petersen graph,
		\item the complete multipartite graph $K_{t \times 3}$ with $t$ parts of cardinality $3$ ($t \ge 2$),
		\item the M\"obius ladder graph on eight vertices,
		\item the Paley graph on 9 vertices.
	\end{enumerate}
In what follows we will therefore assume that $d\ge 3$.
\end{remark}
Let $\G$ be a NDB graph with diameter $d \ge 3$ and with  $\gamma = \gamma(\G) = d+1$.  Pick vertices $x_0, x_d$ of $\G$ such that $d(x_0,x_d)=d$, and let $x_0, x_1, \ldots, x_d$ be a shortest path between $x_0$ and $x_d$. Consider the edge $x_0 x_1$ and note that 
$$
  \{x_1, x_2, \ldots, x_d \} \subseteq W_{x_1,x_0}.
$$
It follows that  there is a unique vertex $u \in W_{x_1,x_0} \setminus \{x_1, x_2, \ldots, x_d \}$. Let $\ell=\ell(x_0, x_1) \; (2 \le \ell \le d)$ be such that $u \in D^{\ell-1}_\ell(x_1,x_0)$, and so $D^{\ell-1}_\ell(x_1,x_0) = \{u,x_\ell\}$ and $D^{i-1}_i(x_1,x_0)=\{x_i\}$ for $2 \le i \le d, i \ne \ell$.


\section{Some structural results}
\label{sec:struc}

Let $\G$ denote a NDB graph with diameter $d \ge 3$ and $\gamma=\gamma(\G)=d+1$. In this section we prove certain structural results about  $\G$. To do this, let us pick arbitrary vertices $x_0, x_d$ of $\G$ with $d(x_0,x_d)=d$,  set $D^i_j=D^i_j(x_1, x_0)$, and let us pick a shortest path $x_0, x_1, \ldots, x_d$ between $x_0$ and $x_d$. Let $\ell=\ell(x_0,x_1)$ and recall that the unique vertex $u \in W_{x_1,x_0} \setminus \{x_1, x_2, \ldots, x_d \}$ is contained in $D^{\ell-1}_{\ell}$. Observe that  
\begin{equation}
	\label{eq1}
	\{x_0, x_1, \ldots, x_{d-1}\} \subseteq W_{x_{d-1},x_d}
\end{equation}
and 
\begin{equation}
	\label{eq2}
	\{x_2, x_3, \ldots, x_d\} \subseteq W_{x_2,x_1}.
\end{equation}
Note that if $\ell \ge 3$, then also $u \in W_{x_2,x_1}$.
In addition, we will use the following abbreviations:
$$
A = \bigcup_{i=2}^{d} \big( \G(x_i) \cap D^i_i\big),
$$
$$
  B = \big( \G(x_2) \cap D^2_1 \big) \cup \big( \G(x_d) \cap D^d_{d-1} \big).
$$ 

\begin{proposition}
\label{p1}
With the notation above, the following {\rm (i), (ii)} hold.
\begin{enumerate}[label={\rm(\roman*)}]
\item There are no edges between $x_i$ and $D^i_{i-1} \cup D^{i-1}_{i-1}$ for $3 \le i \le d-1$.
\item $|\G(x_2) \cap (D^1_1 \cup D^2_1)| \le 1$.
\end{enumerate}
\end{proposition}
\begin{proof}
	(i) Assume that for some $3 \le i \le d-1$ we have that $z$ is a neighbour of $x_i$ contained in $D^i_{i-1} \cup D^{i-1}_{i-1}$. Let $x_0, y_1, \ldots, y_{i-2}, z$ be a shortest path between $x_0$ and $z$. Observe that $\{y_1, \ldots, y_{i-2}, z\} \cap \{x_0, x_1, \ldots x_{d-1}\} = \emptyset$ and that $\{y_1, \ldots, y_{i-2}, z\} \subseteq W_{x_{d-1},x_d}$. These comments, together with \eqref{eq1}, yield $|W_{x_{d-1},x_d}|\ge d+2$, which contradicts the fact that $\gamma=d+1$.
	
	\smallskip \noindent
	(ii) Let $z_1, z_2 \in \G(x_2) \cap (D^1_1 \cup D^2_1)$, $z_1 \ne z_2$. Then $z_1, z_2 \in W_{x_{d-1},x_d}$. This, together with \eqref{eq1}, contradicts the fact that $\gamma=d+1$.
\end{proof}

\begin{proposition}
	\label{p2}
	With the notation above, the following {\rm (i), (ii)} hold.
\begin{enumerate}[label={\rm(\roman*)}]
\item The number $|A \cup B| \le 2$. 
\item If $\ell\ge 3$, then $|A \cup B \cup \big(\G(u) \cap (D^{\ell}_{\ell} \cup D^{\ell}_{\ell-1})\big)| =1$. 
\end{enumerate}
\end{proposition}
\begin{proof}
	(i) Note that $A \cup B \subseteq W_{x_2,x_1}$ and that $(A \cup B) \cap \{x_2, \ldots, x_d\} = \emptyset$. This, together with \eqref{eq2}, forces $|A \cup B| \le 2$.
	
	\smallskip \noindent
	(ii) Note that in this case we have that $u \in W_{x_2, x_1}$. The proof that  $|A \cup B \cup \big(\G(u) \cap (D^{\ell}_{\ell} \cup D^{\ell}_{\ell-1})\big)| \le 1$ is now similar to the proof of (i) above. On the other hand, if  $|A \cup B \cup \big(\G(u) \cap (D^{\ell}_{\ell} \cup D^{\ell}_{\ell-1})\big)| =0$, then $|W_{x_2, x_1}|=d$, contradicting the fact that $\gamma=d+1$.
\end{proof}


\section{Regular NDB graphs with $\gamma=d+1$}
\label{sec:regular}

Let $\G$ denote a regular NDB graph with valency $k$, diameter $d \ge 3$ and $\gamma=\gamma(\G) = d+1$. In this section we use the results from Section \ref{sec:struc} to find bounds on $k$ and $d$. As in the previous section, let us pick arbitrary vertices $x_0, x_d$ of $\G$ with $d(x_0,x_d)=d$, and let us pick a shortest path $x_0, x_1, \ldots, x_d$ between $x_0$ and $x_d$. Set $D^i_j=D^i_j(x_1, x_0)$ and  $\ell=\ell(x_0,x_1)$.

\begin{theorem}
	\label{tval}
	Let $\G$ denote a regular NDB graph with valency $k$, diameter $d \ge 3$ and $\gamma=d+1$.   Then $k \in \{3,4,5\}$. 
\end{theorem}

\begin{proof}
	First note that a cycle $C_n \; (n \ge 3)$ is NDB with $\gamma(C_n)$ equal to the diameter of $C_n$. Therefore, $k\ge 3$.  
	
	Assume first that $\ell=2$. By Proposition~\ref{p1}(ii), $x_2$ has at most one neighbour in $D_1^1 \cup D^2_1$. Moreover, if  $x_2$ and $u$ are adjacent, then also $u \in W_{x_{d-1}, x_d}$, and so in this case $x_2$ has no neighbours in $D_1^1 \cup D^2_1$. By Proposition~\ref{p2}(i), $x_2$ has at most two neighbours in $D_2^2$. Summarizing the above comments, we have that $x_2$ has at most $3$ neighbours in $D^1_1 \cup D^2_1 \cup D^2_2$ if $x_2$ and $u$ are not adjacent, and at most $2$ neighbours in $D^1_1 \cup D^2_1 \cup D^2_2$ if $x_2$ and $u$ are adjacent. Therefore, $k \le 5$.
	
	Assume next that $ \ell \ge 3$. By Propositions~\ref{p1}(ii) and~\ref{p2}(ii), $x_2$ has at most one neighbour in $D^1_1 \cup D^2_1$, and at most one neighbour in $D^2_2$. Therefore, if $\ell=3$ then $u \sim x_2$ and $k \le 5$. If $\ell>3$ then $u \not\sim x_2$ and $k \le 4$. 
\end{proof}

\begin{theorem}
\label{thm:diameter}
Let $\G$ denote a regular NDB graph with valency $k$, diameter $d \ge 3$ and $\gamma=d+1$.  Then the following {\rm (i)-(iii)} hold.
\begin{enumerate}[label={\rm(\roman*)}]
\item If $k=3$, then $d \in \{3,4,5\}$.
\item If $k=4$, then $d \in \{3,4\}$.
\item If  $k=5$, then $d=3$. 
\end{enumerate}
\end{theorem}

\begin{proof}
(i) Assume that $d \ge 6$ and consider first the case $\ell=2$. Note that by Proposition~\ref{p1}(i)  $x_4$ and $x_5$ have a neighbour in $D^4_4$ and $D^5_5$ respectively. If $x_3$ has a heighbour in $D^3_3$ then this contradicts Proposition~\ref{p2}(i). Therefore, $x_3$ and $u$ are adjacent and so $u \in W_{x_{d-1}, x_d}$. This and \eqref{eq1} implies that $x_2$ has no neighbours in $D^1_1 \cup D^2_1$. If $x_2$ and $u$ are adjacent, then we have that $|W_{u, x_2}| = |W_{x_2, u}|=1$, contradicting $\gamma=d+1$. Therefore, $x_2$ has a neighbour in $D_2^2$, contradicting Proposition~\ref{p2}(i). 

If $\ell=3$, then by Proposition~\ref{p1}(i) vertex $x_5$ has a neighbour in $D^5_5$. By Proposition~\ref{p1}(i)  and Proposition~\ref{p2}(ii), $x_3$ and $x_4 $ are both adjacent with $u$. But then $|W_{u, x_3}| = |W_{x_3, u}|=1$, contradicting $\gamma=d+1$.

If $\ell=d-1$, then by Proposition~\ref{p1}(i) vertex $x_3$ has a neighbour in $D^3_3$.
Proposition~\ref{p1}(i) and Proposition~\ref{p2}(ii) now forces that $x_2$ has a neighbour in $D^1_1$ and that $x_{d-1}$ and $u$ are adjacent. As $|W_{x_{d-1}, x_d}|=d+1$ we have that also $x_d$ and $u$ are adjacent (otherwise $u \in W_{x_{d-1}, x_d}$). But now $|W_{u, x_{d-1}}| = |W_{x_{d-1}, u}|=1$, contradicting $\gamma=d+1$.

If $\ell=d$, then $x_3$ and $x_4$ both have a neighbour in $D^3_3$ and $D^4_4$ respectively, contradicting Proposition~\ref{p2}(ii). 

Assume finally that $ 4 \le \ell \le d-2$.  Similarly as above we see that $x_{\ell}$ and $x_{\ell+1}$ are not both adjacent to $u$, so either $x_\ell$ has a neighbour in $D^\ell_\ell$ or $x_{\ell+1}$ has a neighbour in $D^{\ell+1}_{\ell+1}$ (but not both). Therefore we have that $u \in W_{x_{d-1}, x_d}$, and so $x_2$ has no neighbours in $D^1_1 \cup D^2_1$. Consequently, $x_2$ has a neighbour in $D^2_2$, contradicting Proposition~\ref{p2}(ii). 
	
\smallskip \noindent
(ii) Assume $d \ge 5$. If $\ell=2$, then by Proposition~\ref{p1}(i) vertex $x_3$ has at least one neighbour in $D^3_3$, while vertex $x_4$ has two neighbours in $D^4_4$. However, this contradicts Proposition~\ref{p2}(i).

If $\ell \ge 3$, then again by Proposition~\ref{p1}(i) vertex $x_3$ (vertex $x_4$, respectively)  has at least one neighbour in $D^3_3$ ($D^4_4$, respectively), contradicting Proposition~\ref{p2}(ii).

\smallskip \noindent
(iii) Assume $d \ge 4$. It follows from the proof of Theorem~\ref{tval} that in this case $\ell \in \{2,3\}$ holds. If  $\ell=2$, then by Proposition~\ref{p1}(ii) and since $k=5$, vertex $x_2$ has at least one neighbour in $D^2_2$, while  vertex $x_3$ has at least two neighbours in $D_3^3$. However, this contradicts Proposition~\ref{p2}(i). 

If $\ell \ge 3$, then by Proposition~\ref{p1}(i) vertex $x_3$ has at least two neighbours in $D^3_3$, again contradicting Proposition~\ref{p2}(ii). This shows that $d=3$.
\end{proof}

\begin{proposition}\label{ecc3}
	Let $\G$ denote a regular NDB graph with valency $k$, diameter $d = 3$ and $\gamma=4$. Then for every $x \in V(\G)$ we have eccentricity $\epsilon(x)=3$. 
\end{proposition}

\begin{proof}
	By Theorem~\ref{tval}, we know that $k\in \left\lbrace 3,4,5\right\rbrace $. Assume to the contrary that there exists a vertex $x\in V(\G)$ such that $\epsilon(x)\leq 2$. If $\epsilon(x)=1$ then we observe $|V(\G)|=k+1$, and so $\G$ is isomorphic to the complete graph $K_{k+1}$, contradincting $d=3$. Assume next that $\epsilon(x)=2$. Since  $d=3$ there exists $y\in V(\G)$ such that $\epsilon(y)=3$. As $\G$ is connected we may assume that $x\in \G(y)$. Since  $\epsilon(x)=2$ we observe the sets $D^3_2(x,y)$ and $D^3_3(x,y)$ are both empty. Recall that $\gamma=4$, and so by Lemma \ref{eq} we thus have $|D^1_2(x,y)|=|D^2_1(x,y)|=3$, which implies $D^2_3(x,y)=\emptyset$. This contradicts $\epsilon(y)=3$, and so the result follows.
\end{proof}

\begin{proposition}
\label{nonempty}
Let $\G$ denote a regular NDB graph with valency $k$, diameter $d = 3$ and $\gamma=4$. Then for every edge $xy \in E(\G)$ we have that $|D^2_3(x,y)|=|D^3_2(x,y)| \ne 0$.
\end{proposition}

\begin{proof}
	Let us pick an edge $xy \in E(\G)$ . Recall that by Lemma \ref{eq} we have that $|D^1_2(x,y)|=|D^2_1(x,y)|$, and so it follows from \eqref{er} that $|D^2_3(x,y)|=|D^3_2(x,y)|$ as well. Therefore, it remains to prove that the sets $D^2_3(x,y)$ and $D^3_2(x,y)$ are nonempty. 
	
	Assume to the contrary that the sets $D^3_2(x,y)$ and $D^2_3(x,y)$ are empty. As $\gamma=d+1=4$ we have that $|D^1_2(x,y)|=|D^2_1(x,y)|=3$. In view of Theorem~\ref{tval} we therefore have $k\in \left\lbrace 4,5\right\rbrace $. Moreover, 
	by Proposition~\ref{ecc3} the set $D^3_3(x,y)$ is nonempty. Pick $z\in D^3_3(x,y)$ and note that there exists a vertex $w\in \G(z)\cap D^2_2(x,y)$. 
	
	Assume first that $k=4$. Then the set  $D^1_1(x,y)$ is empty. Hence, there exist vertices $u \in D^{1}_2(x,y)$ and $v \in D^2_1(x,y)$ which are neighbours of $w$. We thus have $\left\lbrace u,v,w,x,y \right\rbrace \subseteq W_{w,z}$, contradicting $\gamma=4$. 
	
	Assume next that $k=5$. Note that in this case $|D^1_1(x,y)|=1$. Let us denote the unique vertex of $D^1_1(x,y)$ by $u$. If $w$ and $u$ are not adjacent, then a similar argument as in the previous paragraph shows that $|W_{w,z}|\ge 5$, a contradiction. Therefore, $w$ and $u$ are adjacent, and so $W_{w,z}=\{x,y,u,w\}$. It follows that the remaining three neighbours of $w$ (let us denote these neighbours by $v_1, v_2, v_3$) are also adjacent to $z$. As $\{u,w,z\} \subseteq W_{u,x}$, at least two of these three common neighbours (say $v_1$ and $v_2$) are in $D^2_2$ (recall $D^2_3$ and $D^3_2$ are empty). By the same argument as above (that is $\G(v_1)\cap(D^1_2\cup D^2_1)=\emptyset$ and $\G(v_2)\cap(D^1_2\cup D^2_1)=\emptyset$), $v_1$ and $v_2$ are adjacent to $u$, and so $\{u,w,v_1,v_2,z\} \subseteq W_{u,x}$, a contradiction. This shows that $D^2_3(x,y)$ and $D^3_2(x,y)$ are both nonempty. 
\end{proof}



{\small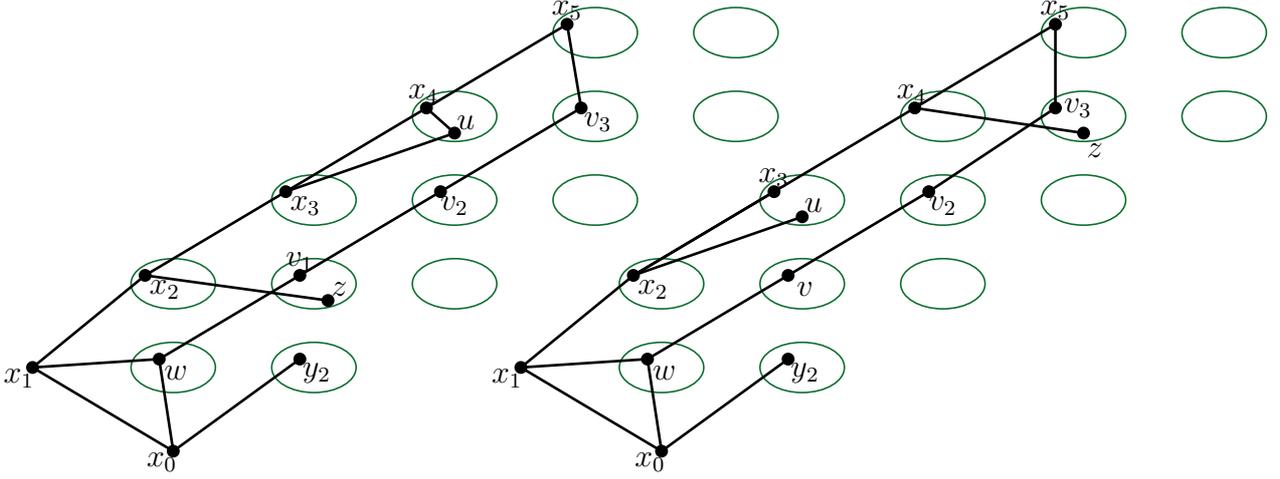
\begin{figure}[t]{\rm
\begin{center}
\begin{tikzpicture}[scale=.37]

\draw[fill=white, draw=ForestGreen, line width=0.6pt] (23.,5.) ellipse (1.5cm and .9cm);
\draw[fill=white, draw=ForestGreen, line width=0.6pt] (18.,5.) ellipse (1.5cm and .9cm);
\draw[fill=white, draw=ForestGreen, line width=0.6pt] (13.,2.) ellipse (1.5cm and .9cm);
\draw[fill=white, draw=ForestGreen, line width=0.6pt] (18.,2.) ellipse (1.5cm and .9cm);
\draw[fill=white, draw=ForestGreen, line width=0.6pt] (23.,2.) ellipse (1.5cm and .9cm);
\draw[fill=white, draw=ForestGreen, line width=0.6pt] (8.,-1.) ellipse (1.5cm and .9cm);
\draw[fill=white, draw=ForestGreen, line width=0.6pt] (13.,-1.) ellipse (1.5cm and .9cm);
\draw[fill=white, draw=ForestGreen, line width=0.6pt] (18.,-1.) ellipse (1.5cm and .9cm);
\draw[fill=white, draw=ForestGreen, line width=0.6pt] (8.,-4.) ellipse (1.5cm and .9cm);
\draw[fill=white, draw=ForestGreen, line width=0.6pt] (13.,-4.) ellipse (1.5cm and .9cm);
\draw[fill=white, draw=ForestGreen, line width=0.6pt] (8.,-7.) ellipse (1.5cm and .9cm);
\draw[fill=white, draw=ForestGreen, line width=0.6pt] (3.,-7.) ellipse (1.5cm and .9cm);
\draw[fill=white, draw=ForestGreen, line width=0.6pt] (3.,-4.) ellipse (1.5cm and .9cm);

\fill (-2,-7) circle [radius=0.23];
\fill (3,-10) circle [radius=0.23];
\node at (-2.1,-7,1) {\normalsize $x_1$};
\node at (2.6,-10.4) {\normalsize $x_0$};

\fill (2,-3.7) circle [radius=0.23]; 
\fill (7,-0.7) circle [radius=0.23]; 
\fill (12,2.3) circle [radius=0.23];  
\fill (13,1.4) circle [radius=0.23]; 
\fill (17,5.3) circle [radius=0.23]; 
\node at (2.7,-4.2) {\normalsize $x_2$};
\node at (7.7,-1.2) {\normalsize $x_3$};
\node at (11.9,2.8) {\normalsize $x_4$};
\node at (13.4,1.8) {\normalsize $u$};
\node at (17,5.8) {\normalsize $x_5$};

\fill (8.5,-4.6) circle [radius=0.23]; 
\node at (8.9,-4.2) {\normalsize $z$};

\draw [line width=1pt, draw=black] (7,-0.7)-- (12,2.3);
\draw [line width=1pt, draw=black] (17,5.3)-- (12,2.3);
\draw [line width=1pt, draw=black] (7,-0.7)-- (13,1.4);
\draw [line width=1pt, draw=black] (12,2.3)-- (13,1.4);

\draw [line width=1pt, draw=black] (3,-10)-- (7.5,-6.7);
\draw [line width=1pt, draw=black] (3,-10)-- (2.5,-6.7);
\draw [line width=1pt, draw=black] (3,-10)-- (-2,-7);
\draw [line width=1pt, draw=black] (2.5,-6.7)-- (-2,-7);
\draw [line width=1pt, draw=black] (2,-3.7)-- (-2,-7);
\draw [line width=1pt, draw=black] (2,-3.7)-- (7,-0.7);
\draw [line width=1pt, draw=black] (2.5,-6.7)-- (7.5,-3.7);
\draw [line width=1pt, draw=black] (12.5,-0.7)-- (7.5,-3.7);
\draw [line width=1pt, draw=black] (12.5,-0.7)-- (17.5,2.3);
\draw [line width=1pt, draw=black] (17,5.3)-- (17.5,2.3);
\draw [line width=1pt, draw=black] (2,-3.7)-- (8.5,-4.6);

\fill (2.5,-6.7) circle [radius=0.23]; 
\fill (7.5,-6.7) circle [radius=0.23]; 
\node at (3.1,-7.2) {\normalsize $w$};
\node at (8.1,-7.2) {\normalsize $y_2$};

\fill (7.5,-3.7) circle [radius=0.23];  
\node at (7.5,-3.2) {\normalsize $v_1$};
\fill (12.5,-0.7) circle [radius=0.23];  
\node at (13,-1.2) {\normalsize $v_2$};
\fill (17.5,2.3) circle [radius=0.23];  
\node at (18.1,1.8) {\normalsize $v_3$};

\end{tikzpicture}~\hspace{-40mm}\begin{tikzpicture}[scale=.37]


\draw[fill=white, draw=ForestGreen, line width=0.6pt] (23.,5.) ellipse (1.5cm and .9cm);
\draw[fill=white, draw=ForestGreen, line width=0.6pt] (18.,5.) ellipse (1.5cm and .9cm);
\draw[fill=white, draw=ForestGreen, line width=0.6pt] (13.,2.) ellipse (1.5cm and .9cm);
\draw[fill=white, draw=ForestGreen, line width=0.6pt] (18.,2.) ellipse (1.5cm and .9cm);
\draw[fill=white, draw=ForestGreen, line width=0.6pt] (23.,2.) ellipse (1.5cm and .9cm);
\draw[fill=white, draw=ForestGreen, line width=0.6pt] (8.,-1.) ellipse (1.5cm and .9cm);
\draw[fill=white, draw=ForestGreen, line width=0.6pt] (13.,-1.) ellipse (1.5cm and .9cm);
\draw[fill=white, draw=ForestGreen, line width=0.6pt] (18.,-1.) ellipse (1.5cm and .9cm);
\draw[fill=white, draw=ForestGreen, line width=0.6pt] (8.,-4.) ellipse (1.5cm and .9cm);
\draw[fill=white, draw=ForestGreen, line width=0.6pt] (13.,-4.) ellipse (1.5cm and .9cm);
\draw[fill=white, draw=ForestGreen, line width=0.6pt] (8.,-7.) ellipse (1.5cm and .9cm);
\draw[fill=white, draw=ForestGreen, line width=0.6pt] (3.,-7.) ellipse (1.5cm and .9cm);
\draw[fill=white, draw=ForestGreen, line width=0.6pt] (3.,-4.) ellipse (1.5cm and .9cm);

\fill (-2,-7) circle [radius=0.23];
\fill (3,-10) circle [radius=0.23];

\node at (-2.1,-7,1) {\normalsize $x_1$};
\node at (2.6,-10.4) {\normalsize $x_0$};

\fill (2,-3.7) circle [radius=0.23]; 

\fill (7,-0.7) circle [radius=0.23]; 
\node at (7,-0.2) {\normalsize $x_3$};

\fill (8,-1.6) circle [radius=0.23]; 
\node at (8.4,-1.2) {\normalsize $u$};

\fill (17,5.3) circle [radius=0.23]; 
\node at (17,5.8) {\normalsize $x_5$};

\node at (2.7,-4.2) {\normalsize $x_2$};

\draw [line width=1pt, draw=black] (2,-3.7)-- (7,-0.7);
\draw [line width=1pt, draw=black] (2,-3.7)-- (8,-1.6);

\fill (12,2.3) circle [radius=0.23];  
\node at (11.9,2.8) {\normalsize $x_4$};

\fill (17,2.3) circle [radius=0.23];  
\node at (17.8,2.3) {\normalsize $v_3$};
\fill (18,1.4) circle [radius=0.23];  
\node at (18.4,0.8) {\normalsize $z$};

\draw [line width=1pt, draw=black] (12,2.3)-- (18,1.4);
\draw [line width=1pt, draw=black] (17,5.3)-- (17,2.3);

\draw [line width=1pt, draw=black] (3,-10)-- (7.5,-6.7);
\draw [line width=1pt, draw=black] (3,-10)-- (2.5,-6.7);
\draw [line width=1pt, draw=black] (3,-10)-- (-2,-7);
\draw [line width=1pt, draw=black] (2.5,-6.7)-- (-2,-7);
\draw [line width=1pt, draw=black] (2,-3.7)-- (-2,-7);
\draw [line width=1pt, draw=black] (2,-3.7)-- (7,-0.7);

\draw [line width=1pt, draw=black] (12,2.3)-- (7,-0.7);
\draw [line width=1pt, draw=black] (12,2.3)-- (17,5.3);
\draw [line width=1pt, draw=black] (2.5,-6.7)-- (7.5,-3.7);
\draw [line width=1pt, draw=black] (12.5,-0.7)-- (7.5,-3.7);
\draw [line width=1pt, draw=black] (12.5,-0.7)-- (17,2.3);

\fill (2.5,-6.7) circle [radius=0.23]; 
\fill (7.5,-6.7) circle [radius=0.23]; 
\node at (3.1,-7.2) {\normalsize $w$};
\node at (8.1,-7.2) {\normalsize $y_2$};
\fill (7.5,-3.7) circle [radius=0.23];  
\node at (8.1,-4.2) {\normalsize $v$};
\fill (12.5,-0.7) circle [radius=0.23];  
\node at (13,-1.2) {\normalsize $v_2$};
\end{tikzpicture}
\caption{\rm 
(a) Case $d=5$, $k=3$ and $\ell=4$ (left).
(b) Case $d=5$, $k=3$ and $\ell=3$ (right).
}
\label{04}
\end{center}
}\end{figure}}





\section{Case $k=3$}
\label{sec:k=3}

Let $\G$ denote a regular NDB graph with valency $k=3$, diameter $d \ge 3$ and $\gamma=\gamma(\G)=d+1$. Recall that by Theorem~\ref{thm:diameter}(i) we have $d \in \{3,4,5\}$. In this section we first show that in fact $d=4$ or $d=5$ is not possible, and then classify $NDB$ graphs with $k=d=3$. We start with a proposition which claims that $d \ne 5$. Although the proof of this proposition is rather tedious and lengthy, it is in fact pretty straightforward. 

\begin{proposition}
	\label{prop:k3dne5}
	Let $\G$ denote a regular NDB graph with valency $k=3$, diameter $d \ge 3$ and $\gamma=\gamma(\G)=d+1$.  Then $d \ne 5$. 
\end{proposition}
\begin{proof}
	Assume to the contrary that $d=5$. Pick vertices $x_0, x_5$ of $\G$ such that $d(x_0,x_5)=5$. Pick also a shortest path $x_0, x_1, x_2, x_3, x_4, x_5$ from $x_0$ to $x_5$ in $\G$. Let $D^i_j=D^i_j(x_1,x_0)$, let $\ell=\ell(x_0,x_1)$ and recall that $2 \le \ell \le 5$.  Observe that if $\ell \ge 3$, then there is a unique vertex $w \in D_1^1$ and a unique vertex $y_2 \in D^2_1$. In this case $x_2$ and $w$ are not adjacent, otherwise edge $w x_1$ is not balanced. Similarly we could prove that $w$ and $y_2$ are not adjacent, and so $w$ has a neighbour $v$ in $D^2_2$. 
	
	Assume first that $\ell=5$. Then by Proposition~\ref{p1}(i) vertex $x_3$ has exactly one neighbour in $D^3_3$. Now vertex $x_2$ has a neighbour in $D^2_1 \cup D^2_2$, contradicting Proposition~\ref{p2}(ii).
	
	Assume $\ell=4$. As $x_2$ has a neighbour in $D^2_1 \cup D^2_2$, Propositions \ref{p1}(i) and \ref{p2}(ii) imply that $x_4$ is adjacent to $u$. If $x_5$ is adjacent to $u$, then $W_{u, x_4} = \{u\}$, a contradiction. Therefore, $x_5$ and $u$ are not adjacent, and so $W_{x_4, x_5} = \{x_4, x_3, x_2, x_1, x_0,u\}$. Consequently, $w \not \in W_{x_4, x_5}$, which implies $d(x_5,w)=4$. It follows that there exists a path $w,v_1,v_2,v_3,x_5$ of length $4$, and it is easy to see that $v_1=v$, $v_2 \in D^3_3$ and $v_3 \in D^4_4$, see Figure~\ref{04}(a).
	
If $x_2$ is adjacent with $y_2$, then $y_2 \in W_{x_4, x_5}$, a contradiction. Therefore, $x_2$ has a neighbour  $z \in D^2_2$. If $z = v$, then $\{x_2, x_3, x_4, x_5, u, v,v_2,v_3\} \subseteq W_{x_2, x_1}$, a contradiction. Therefore $z \ne v$, $W_{x_2, x_1}= \{x_2, x_3, x_4, x_5, u, z\}$, and $z$ is adjacent to $y_2$ (recall that $z$ must be at distance $2$ from $x_0$ and that $y$ is not adjacent with $x_1$ and $v$).
If $z$ has a neighbour in $D^3_2 \cup D^3_3$, then this neighbour would be another vertex in $W_{x_2, x_1}$, which is not possible. The only other possible neighbour of $z$ is $v$, and so $z$ and $v$ are adjacent. It is now clear that $W_{w,v}=\{w,x_0,x_1\}$, contradicting $\gamma=6$.
	 	
{Assume $\ell=3$. By Proposition~\ref{p1}(i), we have that either $x_4$ is adjacent to $u$, or that $x_4$ has a neighbour in $D^4_4$.  Let us first consider the case when $x_4$ and $u$ are adjacent. If also $x_3$ and $u$ are adjacent, then $u x_3$ is clearly not balanced, and so Propositions \ref{p1}(i) and \ref{p2}(ii) imply that $u$ and $x_3$ have a common neighbour $v_2$ in $D^3_3$. Since $x_4 x_5$ is balanced, $v_2$ must be at distance $2$ from $x_5$, which implies that $v_2$ and $x_5$ have a common neighbour $v_3 \in D^4_4$. But now $\{x_2, x_3, x_4, x_5, u, v_2, v_3\} \subseteq W_{x_2, x_1}$, a contradiction. 
Therefore $x_4$ is not adjacent to $u$, and so $x_4$ has a neighbour $z$ in $D^4_4$. Propositions \ref{p1}(i) and \ref{p2}(ii) imply that $x_3$ has no neighbours in $D^2_2 \cup D^3_2 \cup D^3_3$, and so  $x_3$ is adjacent to $u$. This imply that $z$ and $x_5$ are adjacent, as otherwise $x_4 x_5$ is not balanced. Similarly, by Proposition~\ref{p2}(ii) $u$ has no neighbours in $D^3_2 \cup D^3_3$, and so $u$ is adjacent to $v$ (note that $v$ is the unique vertex of $D^2_2$). As in the previous paragraph (since $w\not\in W_{x_4, x_5}=\{x_4,x_3,x_2,x_1,x_0,u\}$) we obtain that there exists a path $w,v,v_2,v_3,x_5$ of length $4$, and that $v_2 \in D^3_3$, $v_3 \in D^4_4$ (note that it could happen that $z=v_3$). Note that $u$ and $x_3$ have no neighbours in $D^3_3$, and that the only neighbour of  $v$ in $D^3_3$ is $v_2$. Therefore, as $k=3$, this implies that $v_2$ is the unique vertex of $D^3_3$. Let us now examine the cardinality of $D^4_4$. By Proposition~\ref{p2}(ii), both neighbours of $x_5$, different from $x_4$, are in $D^4_4$, and so $|D^4_4| \ge 2$. On the other hand, if $v_2$ has two neighbours in $D^4_4$, then $w x_0$ is not balanced, and so $v_3$ is the unique neighbour of $v_2$ in $D^4_4$. As $x_4$ has exactly one neighbour in $D^4_4$ (namely $z$), this shows that $|D^4_4| = 2$ and that $v_3 \ne z$. 
But as $\G$ is a cubic graph, it must have an even order. Then, there exists a vertex $t$ in $D^5_5$. Note that $t$ is not adjacent to $x_5$, and so it must be adjacent to at least one of $z,v_3$. However, if $t$ is adjacent to $z$, then $x_2 x_1$ is not balanced, while if it is adjacent to $v_3$, then $w x_0$ is not balanced. This shows that $\ell \ne 3$}

Assume finally that $\ell=2$. By Proposition~\ref{p1}(i), vertex $x_4$ has a neighbour $z \in D^4_4$. Also by Proposition~\ref{p1}(i), vertex $x_3$ either has a neighbour in $D^3_3$, or is adjacent with $u$. Assume first that $x_3$ is adjacent with $u$. Note that in this case $x_2\not\sim u$ (otherwise edge $x_2u$ is not balanced) and $\{x_4, x_3, x_2, x_1, x_0, u\} = W_{x_4, x_5}$. It follows that $x_2$ cannot have a neighbour in $D^2_1$ (otherwise the edge $x_4x_5$ is not balanced) and so $x_2$ has a neighbour $v \in D^2_2$. Now if $v$ has a neighbour $v_2 \in D^3_3$, then $\{x_2, x_3, x_4, x_5, z,v, v_2\} \subseteq W_{x_2, x_1}$, a contradiction. Therefore $v$ has no neighbours in $D^3_3$, implying that $d(x_5,v)=4$. But this forces $v \in W_{x_4, x_5}$, a contradiction. Thus $x_3\not\sim u$, and it follows that $x_3$ has a neighbour $v_2 \in D^3_3$. As $\{x_2, x_3, x_4, x_5, v_2,z\}=W_{x_2, x_1}$, vertex $x_2$ has no neighbours in $D^2_1 \cup D^2_2$, implying that $x_2$ is adjacent to $u$. Since $W_{x_4, x_5}=\{x_4,x_3,x_2,x_1,x_0,u\}$, vertex $z$ is adjacent to $x_5$, and vertices $v_2$ and $x_5$ have a common neighbour in $D^4_4$. Now, since $x_1 x_2$ is balanced we have that this common neighbour is in fact $z$, and so $z$ is adjacent to $v_2$. Now consider the edge $v_2 z$. Note that $\{x_1, x_2, x_3, v_2\} \subseteq W_{v_2, z}$. As $d(x_0,v_2)=3$, there exist vertices $y_1, y_2$, such that $x_0, y_1, y_2, v_2$ is a path of length $3$ between $x_0$ and $v_2$. Observe that $\{x_0, y_1, y_2, v_2\} \subseteq W_{v_2, z}$. As $\{x_1, x_2, x_3\} \cap \{x_0, y_1, y_2\}=\emptyset$, we have that $|W_{v_2,z}| \ge 7$, a contradiction.
\end{proof}


\subsection{Case $d=4$ is not possible}

Let $\G$ denote a regular NDB graph with valency $k=3$, diameter $d \ge 3$ and $\gamma=\gamma(\G)=d+1$.  We now consider the case $d=4$. Our main result in this subsection is to prove that this case is not possible. For the rest of this subsection pick arbitrary vertices $x_0, x_4$ of $\G$ such that $d(x_0,x_4)=4$. Pick a shortest path $x_0, x_1, x_2, x_3, x_4$ between $x_0$ and $x_4$. Let $D^i_j=D^i_j(x_1, x_0)$ and let $\ell=\ell(x_0, x_1)$. Let $u$ denote the unique vertex of $D^{\ell-1}_\ell \setminus \{x_\ell\}$.

\begin{proposition}\label{ell4d4}
	Let $\G$ denote a regular NDB graph with valency $k=3$, diameter $d =4$ and $\gamma=\gamma(\G)=d+1=5$. With the notation above, we have that $\ell \ne 4$. 
\end{proposition}

\begin{proof}
	Assume to the contrary that $\ell=4$. Note that in this case since $k=3$ and $|D^1_2|=|D^2_1|=1$ we have $|D^1_1|=1$.  Let $w$ denote the unique vertex of $D^1_1$, and let $z$ denote the neighbour of $x_2$, different from $x_1$ and $x_3$. Observe that $z \ne w$, as otherwise $x_1 w$ is not balanced. Similarly, $w$ is not adjacent to the unique vertex $y_2$ of $D^2_1$. Observe also that $\left\lbrace x_0, x_1, x_2, x_3\right\rbrace \subseteq W_{x_3, u} $. We claim that $u \in \G(x_4)$. To prove this, suppose that $x_4$ and $u$ are not adjacent. Then $x_4 \in W_{x_3, u}$, and so $z$ is contained in $D^2_2$. Observe that $d(z,u)=2$, otherwise $x_3 u$ is not balanced. Therefore, $u$ and $z$ must have a common neighbour $z_1$ and it is clear that $z_1 \in D_3^3$. But now $\{x_2, x_3, x_4, u, z, z_1\} \subseteq W_{x_2, x_1}$, a contradiction. This proves our claim that $u\sim z$. 
	
	 Suppose now that $z=y_2$. Then $D^3_2 \cup D^4_3 \cup \left\lbrace u, x_2, x_3, x_4,y_2 \right\rbrace \subseteq W_{x_2, x_1}$. Note that by the NDB condition we have $|D^3_2\cup D^4_3|=3$, and so $x_2 x_1$ is not balanced, a contradiction. We therefore have that $z \in D^2_2$. 
	 
	 By Proposition~\ref{p2}(ii) it follows that $u$ and $x_4$
	 have a neighbour $z_1$ and $z_2$ in $D^3_3$, respectively. We observe $z_1 \ne z_2$, as otherwise $x_4 u$ is not balanced. Note that $z$ has no neighbours in $D^3_3$, as otherwise $x_2 x_1$ is not balanced. Therefore, $z $ is not adjacent to any of $z_1, z_2$, which gives us $W_{x_3, x_4} = W_{x_3, u} = \{x_3, x_2, x_1, x_0,z\}$. Consequently, $d(w,u)=d(w,x_4)=3$, and so the (unique) neighbour of $w$ in $D^2_2$ is adjacent to both $z_1$ and $z_2$. But this implies that $wx_0$ is not balanced, a contradiction.
\end{proof}

\begin{proposition}\label{ell3d4}
	Let $\G$ denote a regular NDB graph with valency $k=3$, diameter $d =4$ and $\gamma=\gamma(\G)=d+1=5$. With the notation above, we have that $\ell \ne 3$. 
\end{proposition}

\begin{proof}
	Suppose that $\ell=3$. By Lemma \ref{eq} we have $|D^2_1|=1$, and since $k=3$ also $|D^1_1|=1$. Let $w$ and $y_2$  denote the unique vertex of $D^1_1$ and $D^2_1$, respectively. Since $\gamma=5$, $y_2$ has at least one neighbour $y_3$ in $D^3_2$, and $|D^4_3| \le 2$. If $D^4_3=\emptyset$, then there are three vertices in $D^3_2$, which are all adjacent to $y_2$, contradicting $k=3$. By  Proposition \ref{ell4d4} we have that $|D^4_3| \ne 2$, and so $|D^4_3| =1$, $|D^3_2| = 2$. Let $y_4$ denote the unique element of $D^4_3$ and let $u_1$ denote the unique element of $D^3_2 \setminus \{y_3\}$. Without loss of generality assume that $y_4$ and $y_3$ are adjacent. Observe that $\G(y_2)=\{x_0,y_3,u_1\}$, and so $w$ has a neighbour $v \in D^2_2$, and it is easy to see that $v$ is the unique vertex of $D^2_2$ (see Figure~\ref{07}(a)). By Proposition~\ref{p1}(i) we find that  either $x_3 \in \G(u)$, or $x_3$ has a neighbour in $D^3_3$. 
	
	{\sc Case 1:} there exists $z \in \G(x_3)\cap D^3_3$. Note that in this case we have $W_{x_2, x_1} = \{x_2, x_3, x_4, u, z\}$. We split our analysis into two subcases.
	
	{\sc Subcase 1.1:} vertices $u$ and $x_4$ are not adjacent. As $x_2x_1$ is balanced and as $v$ is the unique vertex of $D^2_2$, this forces that $u$ is adjacent with $v$ and $z$. As every vertex in $D^3_3$ is at distance $3$ from $x_1$ and as vertices $u$, $x_3$ already have three neighbours each, this imply that beside $z$ there is at most one more vertex in $D^3_3$ (which must be adjacent with $v$). But this shows that $x_4$ could have at most one neighbour in $D^3_3$ (observe that $z$ could not be adjacent with $x_4$, as otherwise $z$ is not at distance $3$ from $x_0$), and consequently $x_4$ has at least one neighbour in $D^4_4 \cup D^4_3$. But now $x_2x_1$ is not balanced, a contradiction.
	
	{\sc Subcase 1.2:} {vertices $u$ and $x_4$ are adjacent. By Proposition~\ref{p2}(ii), vertex $u$ is either adjacent to $v \in D^2_2$ or to $z \in D^3_3$. If $u$ is adjacent to $v$, then $\{ x_0, x_1, x_2, u, v, w \} \subseteq W_{u, x_4}$, a contradiction.  This shows that $u \sim z$. Note that the third neighbour of $z$ is one of the vertices $v,y_3, u_1$, and so $z$ and $x_4$ are not adjacent. Consequently, $W_{x_3, x_4}=\{x_3, x_2, x_1, x_0, z\}$, and so $w$ must be at distance $3$ from $x_4$. Therefore, $v$ and $x_4$ have a common neighbour $v_1 \in D^3_3$. Note that $v_1 \ne z$ as $z$ and $x_4$ are not adjacent. Every vertex in $D^3_3$, different from $z$ and $v_1$, must be adjacent with $v$ in order to be at distance $3$ from $x_1$. This shows that $|D^3_3| \le 3$. If there exists vertex $v_2 \in D^3_3$, which is different from $z$ and $v_1$, then there must be a vertex $t \in D^4_4$ (recall that $\G$ is of even order). As $t$ could not be adjacent with $x_4$, it must be adjacent with at least one of $v_1, v_2$. However, this is not possible (note that in this case $\{w,v,v_1,v_2,x_4,t\} \subseteq W_{w,x_0}$, a contradiction). Therefore, $D^3_3 = \{z,v_1\}$ and $D^4_4 = \emptyset$. It follows that $y_4$ is adjacent with $v_1$ and $u_1$. If $z$ and $v$ are adjacent, then $W_{x_1,w}=\{x_1,x_2,u,x_3\}$, contradicting $\gamma=5$. Therefore, $z$ is adjacent to either $y_3$ or $u_1$.  This shows that either $y_3$ or $u_1$ is contained in $W_{x_3, x_4}=\{x_3, x_2, x_1, x_0, z\}$, a contradiction. }

{\small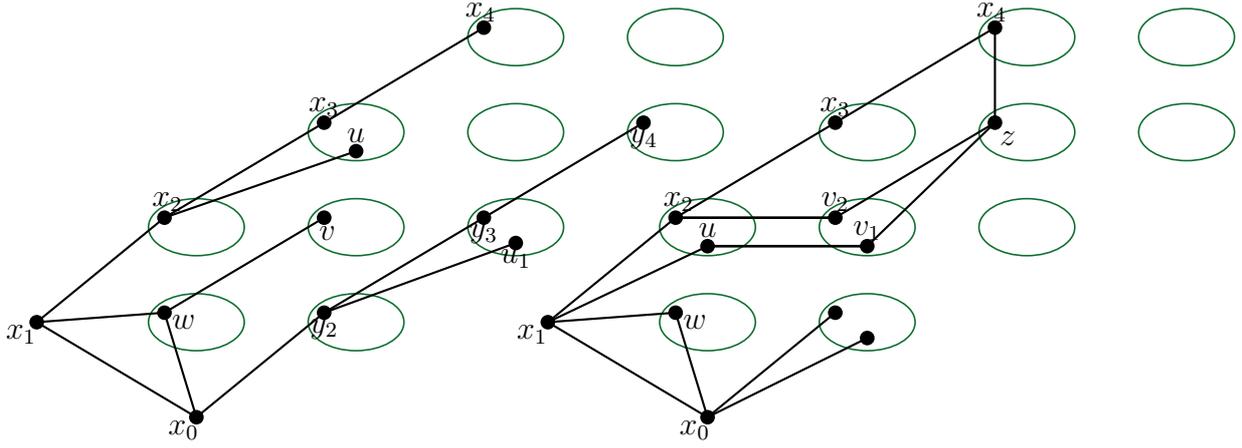
\begin{figure}[t]{\rm
\begin{center}
\begin{tikzpicture}[scale=.42]
\draw[fill=white, draw=ForestGreen, line width=0.6pt] (13.,2.) ellipse (1.5cm and .9cm);
\draw[fill=white, draw=ForestGreen, line width=0.6pt] (18.,2.) ellipse (1.5cm and .9cm);
\draw[fill=white, draw=ForestGreen, line width=0.6pt] (8.,-1.) ellipse (1.5cm and .9cm);
\draw[fill=white, draw=ForestGreen, line width=0.6pt] (13.,-1.) ellipse (1.5cm and .9cm);
\draw[fill=white, draw=ForestGreen, line width=0.6pt] (18.,-1.) ellipse (1.5cm and .9cm);
\draw[fill=white, draw=ForestGreen, line width=0.6pt] (8.,-4.) ellipse (1.5cm and .9cm);
\draw[fill=white, draw=ForestGreen, line width=0.6pt] (13.,-4.) ellipse (1.5cm and .9cm);
\draw[fill=white, draw=ForestGreen, line width=0.6pt] (8.,-7.) ellipse (1.5cm and .9cm);
\draw[fill=white, draw=ForestGreen, line width=0.6pt] (3.,-7.) ellipse (1.5cm and .9cm);
\draw[fill=white, draw=ForestGreen, line width=0.6pt] (3.,-4.) ellipse (1.5cm and .9cm);

\fill (-2,-7) circle [radius=0.23]; 
\fill (3,-10) circle [radius=0.23];

\node at (-2.1,-7,1) {\normalsize $x_1$};
\node at (2.6,-10.4) {\normalsize $x_0$};

\draw [line width=0.8pt, draw=black] (-2,-7)-- (2,-3.7);
\draw [line width=0.8pt, draw=black] (7,-0.7)-- (2,-3.7);
\draw [line width=0.8pt, draw=black] (8,-1.6)-- (2,-3.7);
\draw [line width=0.8pt, draw=black] (7,-0.7)-- (12,2.3);

\draw [line width=0.8pt, draw=black] (3,-10)-- (7,-6.7);
\draw [line width=0.8pt, draw=black] (3,-10)-- (2,-6.7);
\draw [line width=0.8pt, draw=black] (-2,-7)-- (2,-6.7);
\draw [line width=0.8pt, draw=black] (12,-3.7)-- (7,-6.7);
\draw [line width=0.8pt, draw=black] (2,-6.7)-- (7,-3.7);
\draw [line width=0.8pt, draw=black] (13,-4.5)-- (7,-6.7);
\draw [line width=0.8pt, draw=black] (12,-3.7)-- (17,-0.7);
\draw [line width=0.8pt, draw=black] (3,-10)-- (-2,-7);

\fill (7,-0.7) circle [radius=0.23]; 
\node at (7,-0.2) {\normalsize $x_3$};

\fill (8,-1.6) circle [radius=0.23]; 
\node at (8,-1.1) {\normalsize $u$};

\fill (12,2.3) circle [radius=0.23];  
\node at (11.9,2.8) {\normalsize $x_4$};

\fill (2,-3.7) circle [radius=0.23]; 
\node at (2.1,-3.2) {\normalsize $x_2$};

\fill (2,-6.7) circle [radius=0.23]; 
\node at (2.6,-7) {\normalsize $w$};

\fill (7,-3.7) circle [radius=0.23]; 
\node at (7.1,-4.2) {\normalsize $v$};

\fill (7,-6.7) circle [radius=0.23]; 
\node at (7,-7.2) {\normalsize $y_2$};

\fill (12,-3.7) circle [radius=0.23]; 
\node at (12,-4.2) {\normalsize $y_3$};

\fill (13,-4.5) circle [radius=0.23]; 
\node at (13,-5) {\normalsize $u_1$};

\fill (17,-0.7) circle [radius=0.23]; 
\node at (17,-1.2) {\normalsize $y_4$};

\end{tikzpicture}~\hspace{-30mm}\begin{tikzpicture}[scale=.42]
\draw[fill=white, draw=ForestGreen, line width=0.6pt] (13.,2.) ellipse (1.5cm and .9cm);
\draw[fill=white, draw=ForestGreen, line width=0.6pt] (18.,2.) ellipse (1.5cm and .9cm);
\draw[fill=white, draw=ForestGreen, line width=0.6pt] (8.,-1.) ellipse (1.5cm and .9cm);
\draw[fill=white, draw=ForestGreen, line width=0.6pt] (13.,-1.) ellipse (1.5cm and .9cm);
\draw[fill=white, draw=ForestGreen, line width=0.6pt] (18.,-1.) ellipse (1.5cm and .9cm);
\draw[fill=white, draw=ForestGreen, line width=0.6pt] (8.,-4.) ellipse (1.5cm and .9cm);
\draw[fill=white, draw=ForestGreen, line width=0.6pt] (13.,-4.) ellipse (1.5cm and .9cm);
\draw[fill=white, draw=ForestGreen, line width=0.6pt] (8.,-7.) ellipse (1.5cm and .9cm);
\draw[fill=white, draw=ForestGreen, line width=0.6pt] (3.,-7.) ellipse (1.5cm and .9cm);
\draw[fill=white, draw=ForestGreen, line width=0.6pt] (3.,-4.) ellipse (1.5cm and .9cm);

\draw [line width=0.8pt, draw=black] (3,-10)-- (7,-6.7);
\draw [line width=0.8pt, draw=black] (3,-10)-- (8,-7.5);
\draw [line width=0.8pt, draw=black] (3,-10)-- (2,-6.7);
\draw [line width=0.8pt, draw=black] (-2,-7)-- (2,-6.7);
\draw [line width=0.8pt, draw=black] (3,-10)-- (-2,-7);

\draw [line width=0.8pt, draw=black] (-2,-7)-- (2,-3.7);
\draw [line width=0.8pt, draw=black] (7,-0.7)-- (2,-3.7);
\draw [line width=0.8pt, draw=black] (-2,-7)-- (3,-4.6);
\draw [line width=0.8pt, draw=black] (7,-0.7)-- (12,2.3);
\draw [line width=0.8pt, draw=black] (12,-0.7)-- (8,-4.6);
\draw [line width=0.8pt, draw=black] (12,-0.7)-- (7,-3.7);
\draw [line width=0.8pt, draw=black] (12,-0.7)-- (12,2.3);

\fill (2,-6.7) circle [radius=0.23]; 
\node at (2.6,-7) {\normalsize $w$};

\fill (7,-6.7) circle [radius=0.23]; 
\fill (8,-7.5) circle [radius=0.23]; 

\fill (-2,-7) circle [radius=0.23];
\fill (3,-10) circle [radius=0.23];

\node at (-2.1,-7,1) {\normalsize $x_1$};
\node at (2.6,-10.4) {\normalsize $x_0$};

\fill (7,-0.7) circle [radius=0.23]; 
\node at (7,-0.2) {\normalsize $x_3$};

\fill (3,-4.6) circle [radius=0.23]; 
\node at (3,-4.1) {\normalsize $u$};
\fill (8,-4.6) circle [radius=0.23]; 
\node at (8,-4.1) {\normalsize $v_1$};
\draw [line width=1pt, draw=black] (3,-4.6)-- (8,-4.6);

\fill (12,2.3) circle [radius=0.23];  
\node at (11.9,2.8) {\normalsize $x_4$};
\fill (12,-0.7) circle [radius=0.23];  
\node at (12.4,-1.2) {\normalsize $z$};

\fill (2,-3.7) circle [radius=0.23]; 
\node at (2.1,-3.2) {\normalsize $x_2$};
\fill (7,-3.7) circle [radius=0.23];  
\node at (7,-3.2) {\normalsize $v_2$};
\draw [line width=1pt, draw=black] (2,-3.7)-- (7,-3.7);

\end{tikzpicture}
\caption{\rm 
(a) Case $d=4$, $k=3$ and $\ell=3$ (left).
(b) Case $d=4$, $k=4$ and $\ell=2$ (right).
}
\label{07}
\end{center}
}\end{figure}}


{\sc Case 2:} $x_3$ and $u$ are adjacent. Observe that $x_4 \notin \G(u)$, otherwise $u x_3$ is not balanced. It follows that $W_{x_3, x_4}=\{x_3, x_2, x_1, x_0,u\}$, and so $d(w,x_4)=3$. Therefore there exists a common neighbour $z$ of $x_4$ and $v$, and note that $z \in D^3_3$. Reversing the roles of the paths $x_0, x_1, x_2, x_3, x_4$ and $x_1, x_0, y_2, y_3, y_4$, we get that $u_1$ and $y_3$ are adjacent, and that $y_4 \notin \G(u_1)$.  As $|W_{x_1, w}|=5$, vertex $u$ must have a neighbour, which is at distance $3$ from $x_1$ and at distance $4$ from $w$. As $x_4, y_3$ and $u_1$ ar all at distance $3$ from $w$, this implies that $u$ has a neighbour $z_1 \in D^3_3$, which is not adjacent with $v$ (and is therefore different from $z$). Note that since $z_1$ is at distance $3$ from $x_0$, it is adjacent with $u_1$. As $k=3$, $v$ has a neighbour $z_2 \ne z$ in $D^3_3$. Pick now a vertex $t \in D^4_4$ (observe that $D^4_4 \ne \emptyset$ as $\G$ has even order). If $t$ is adjacent with $x_4$ or with $z_1$, then $t \in W_{x_2, x_1} = \{x_2, x_3, x_4, u, z_1\}$, a contradiction. If $t$ is adjacent with $z$ or $z_2$, then $t \in W_{w,x_0}=\{w,v,z,z_2,x_4\}$, a contradiction. This finally proves that $\ell \ne 3$. 
\end{proof}

\begin{proposition}\label{trianfree}
	Let $\G$ denote a regular NDB graph with valency $k=3$, diameter $d =4$ and $\gamma=\gamma(\G)=d+1=5$. With the notation above,  $\G$ is triangle-free.  
\end{proposition}

\begin{proof}
	Pick an edge $xy \in E(\G)$ and let $D^i_j=D^i_j(x,y)$. If either $D^4_3$ or $D^3_4$ is nonempty, then Propositions \ref{ell4d4} and \ref{ell3d4} together with Lemma \ref{eq} imply that $|D^1_2|=|D^2_1|=2$. As $\G$ is $3$-regular, the set $D^1_1$ is empty, and so $xy$ is not contained in any triangle. 
	
	Assume next that $D^4_3=D^3_4=\emptyset$. If the edge $xy$ is contained in a triangle, then $D^1_2$ and $D^2_1$ both contain at most one vertex, and so $D^2_3$ and $D^3_2$ could contain at most two vertices as $\G$ is $3$-regular. We thus have  $|W_{x,y}| \le 4$, contradicting $\gamma=5$. The result follows. 
\end{proof}

\begin{proposition}\label{ell2d4}
	Let $\G$ denote a regular NDB graph with valency $k=3$, diameter $d \ge 3$ and $\gamma=\gamma(\G)=d+1$.  Then $d \ne 4$. 
\end{proposition}

\begin{proof}
	Towards a contradiction suppose that $d=4$, and so $\gamma=5$. Assume the notation from the first paragraph of this subsection, and note that Propositions \ref{ell4d4} and \ref{ell3d4} imply that $\ell=2$. By Lemma \ref{eq} we have $|D^2_1|=2$. Let $u_1, y_2$ denote the vertices of $D^2_1$. Note that $D^1_1$ is empty. We also observe that by Proposition~\ref{p1}(i) either $u\in \G(x_3)$, or $x_3$ has a neighbour in $D^3_3$. We consider these two cases separately. 
	
	{\sc Case 1:} $u$ and $x_3$ are adjacent. Then $\{x_0, x_1, x_2, x_3, u\} = W_{x_3,x_4}$, and so neither $x_2$ nor $u$ have neighbours in $D^2_1$. Since $\G$ is triangle-free, there exists $w \in \G(x_2)\cap D^2_2$, and $w$ has a neighbour in $D^2_1$ (by definition of the set $D^2_2$). We may assume without loss of generality that $w\in \G(y_2)$. Note that $d(w, x_3)=2$, and so $d(w, x_4)=2$ as well, as otherwise $x_3 x_4$ is not balanced. It follows that there exists a common neighbour $z$ of $w$ and $x_4$, and it is clear that $z \in D^3_3$.
	
	Similarly we find that $u$ has a neighbour $w_1 \in D^2_2$, and as $k=3$, we have that $w_1 \ne w$. Note that $\{x_2, x_1, x_0,w,y_2\} = W_{x_2, x_3}$, and so $d(x_3,u_1)=3$ (otherwise $u_1 \in W_{x_2, x_3}$, a contradiction). Note however that $d(x_3,u_1)=3$ is only possible if $w_1$ and $u_1$ are adjacent. A similar argument as above shows that $w_1$ and $x_4$ must have a common neighbour $z_1 \in D^3_3$. If $z_1=z$, then $\{z,w,w_1,y_2,u_1,x_0\} \subseteq W_{z,x_4}$, a contradiction. Therefore $z_1 \ne z$, and it is now clear that $D^2_2=\{w, w_1\}$, $D^3_3 = \{z,z_1\}$. If there exists $t \in D^4_4$, then $t$ is adjacent to either $z$ or $z_1$, but none of these two possible edges is balanced, and so $D^4_4 = \emptyset$. If $z$ ($z_1$, respectively) has a neighbour in $D^4_3$, then $x_2 x_1$ ($u x_1$, respectively) is not balanced, a contradiction. As $\G$ is triangle-free, $z$ and $z_1$ both have a neighbour in $D^3_2$. {Assume now for a moment that there exists a vertex $y_4 \in D^4_3$. In this case $\gamma=5$ forces that there is a unique vertex in $D^3_2$, which is therefore adjacent to both $z$ and $z_1$, to $y_4$ and to at least one of $y_2, u_1$, contradicting $k=3$. It follows that $D^4_3 = \emptyset$. Let us denote the neighbours of $z$ and $z_1$ in $D^3_2$ by $v$ and $v_1$ respectively. Note that as $z x_4$ and $z_1 x_4$ are balanced, we have that  $W_{z,x_4}=\{z,w,v,y_2,x_0\}$ and $W_{z_1,x_4}=\{z_1,w_1,v_1,u_2,x_0\}$. It follows that $v$ and $v_1$ must be adjacent to $y_2$ and $u_1$, respectively, and so $v \ne v_1$. As $k=3$, also $v$ and $v_1$ are adjacent. It is now easy to see that $\G$ is not NDB with $\gamma=5$ (for example, edge $x_1u$ is not balanced). This shows that $u$ and $x_3$ are not adjacent.}
	
	{\sc Case 2:}  $x_3$ has a neighbour $w$ in $D^3_3$. As $\G$ is triangle-free, $x_2$ has a neighbour $z$ in $D^2_1\cup D^2_2$, and $w \not \sim x_4$. If $z \in D^2_1$, then $\{x_0, x_1, x_2, x_3, z, w\} \subseteq W_{x_3, x_4}$, a contradiction. This yields that $z \in D^2_2$.  If $d(z,x_4) \ge 3$, then again $\{x_0, x_1, x_2, x_3, z, w\} \subseteq W_{x_3, x_4}$, a contradiction. Therefore, $z$ and $x_4$ have a common neighbour $w_1 \in D_3^3$, and $w_1 \ne w$ as $w \not \sim x_4$. But now $\{x_2, x_3, x_4, z, w,w_1\} \subseteq W_{x_2, x_1}$, a contradiction. This finishes the proof.
\end{proof}

\subsection{Case $d=3$}

In this subsection we consider the case $d=3$. We start with the following proposition. 

\begin{proposition}\label{ell2k3}
	Let $\G$ denote a regular NDB graph with valency $k=3$, diameter $d=3$ and $\gamma=4$. Then for every edge $x_0 x_1$ of $\G$ we have that $|D^1_2(x_1, x_0)|=|D^2_1(x_1, x_0)|=2$. 
\end{proposition}

\begin{proof}
Pick an edge $x_0 x_1$ of $\G$ and let $D^i_j = D^i_j(x_1, x_0)$. Observe first that that $|D^1_2| \le 2$ as $k=3$. By Proposition~\ref{nonempty} we have that $D^2_3 \ne \emptyset$, and so pick $x_3 \in D^2_3$. Note that $x_1$ and $x_3$ have  a common neighbour $x_2 \in D^1_2$. Assume to the contrary that $|D^1_2|=1$, and so $|D^2_3|=2$, $|D^1_1|=1=|D^2_1|$. Let us denote the unique vertex of $D^2_1$ by $y_2$ (note that $y_2$ has two neighbors, say $y_3$ and $u_1$  in $D^3_2$), the unique vertex of $D^1_1$ by $w$, and the unique vertex of $D^2_3 \setminus \{x_3\}$ by $u$ (note that $\G(x_2)=\{x_1,x_3,u\}$).  Note that $w$ has a neighbour $v$ in $D^2_2$, and that $D^2_2=\{v\}$.
	
Assume first that $u$ and $x_3$ are not adjacent. Then $W_{x_2, x_3} = \{x_2,u,x_1,x_0\}$, and so $w$ is at distance $2$ from $x_3$ (otherwise $w \in W_{x_2, x_3}$). It follows that $x_3$ is adjacent with $v$. Similarly we show that $u$ is adjacent with $v$. As none of the neighbours of $v$ is contained in $D^3_3$, every vertex from $D^3_3$ must be adjacent to either $u$ or $x_3$, and so $D_3^3 \cup \{x_2, x_3, u\}\subseteq W_{x_2, x_1}$. It follows that $|D^3_3| \le 1$. As $\G$ is a cubic graph, it must have an even order, which gives us $D^3_3 = \emptyset$. This shows that both $u$ and $x_3$ have a neighbour in $D^3_2$. But now $\{y_2, y_3,u_1,x_3, u\} \cup D^3_2 \subseteq W_{y_2, x_0}$, a contradiction.
	
Therefore, $u$ and $x_3$ must be adjacent, and they have a common neighbour $x_2$. Let $z_1$ and $z_2$ denote the third neighbour of $u$ and $x_3$, respectively. If $z_1 = z_2$ then $u x_3$ is not balanced, and so we have that $z_1 \ne z_2$. Furthermore, as $\{x_2, x_3, u\} \subseteq W_{x_2, x_1}$, not both of $z_1, z_2$ are contained in $D^3_3 \cup D^3_2$. Therefore, either $z_1$ or $z_2$ is equal to $v$. Without loss of generality assume that $z_1=v$. But then $d=3$ forces $W_{x_2, u} = \{x_2, x_1, x_0\}$, a contradiction. This  shows that $|D^1_2|=2$, and by Lemma \ref{eq} also $|D^2_1|=2$.
\end{proof}

\begin{corollary}
	\label{cor:k=3-d=3}
	Let $\G$ denote a regular NDB graph with valency $k=3$, diameter $d=3$ and $\gamma=4$. Then $\G$ is triangle-free and $D^3_3(x,y)=\emptyset$ for every edge $xy$ of $\G$ 
\end{corollary}
\begin{proof}
	Pick an arbitrary edge $xy$ of $\G$ and let $D^i_j=D^i_j(x,y)$. By Proposition~\ref{nonempty} we get that the sets  $D^1_2$, $D^2_1$, $D^2_3$ and $D^3_2$ are all nonempty. Furthermore,  by Proposition~\ref{ell2k3} and Lemma \ref{eq} we have that $|D^1_2|=|D^2_1|=2$ and $|D^3_2| = |D^2_3|=1$ (recall that $\gamma=4$). Since $k=3$, it follows  that $D^1_1 = \emptyset$. This shows that $\G$ is triangle-free. 
	
	We next assert the set $D^3_3$ is empty. Suppose to the contrary there exists $z\in D^3_3$ and let $w$ denote a neighbour of $z$. Assume first that $w \in D^2_2$. Since $D^1_1=\emptyset$, there exist vertices $u \in D^1_2$ and $v \in D^2_1$ which are neighbours of $w$. We thus have $\left\lbrace u,v,w,x,y \right\rbrace \subseteq W_{w,z}$, contradicting $\gamma=4$. This shows that $w\notin D^2_2$. Therefore $z$ is adjacent to both vertices which are in $D^3_2$ and $D^2_3$. As $z$ has three neighbours, none of which is in $D^2_2$, and as $|D^2_3|=|D^3_2|=1$, it follows that $z$ has a neighbour $w' \in D^3_3$. But by the same argument as above, $w'$ must be adjacent to both vertices in $D^3_2$ and $D^2_3$, contradicting the fact that $\G$ is triangle-free.
\end{proof}

\begin{theorem}
	\label{thm:k3d3}
  Let $\G$ denote a regular NDB graph with valency $k=3$, diameter $d \ge 3$ and $\gamma=d+1$. Then $\G$ is isomorphic to the $3$-dimensional hypercube $Q_3$. 
\end{theorem}

\begin{proof}
	By Theorem~\ref{thm:diameter}(i), Proposition~\ref{prop:k3dne5} and Proposition~\ref{ell2d4} we have that $d =3$. Pick an edge $xy$ of $\G$ and let $D^i_j=D^i_j(x,y)$. Observe that $\G$ is triangle-free and $D^3_3=\emptyset$ by Corollary \ref{cor:k=3-d=3}. We first show that $D^2_2 = \emptyset$ as well. Observe that as $D^1_1=\emptyset$,  every vertex of $D^2_2$ must have a neighbour in both  $D^1_2$ and $D^2_1$. This shows $\left|D^2_2 \right| \in \left\lbrace 1, 2, 3 \right\rbrace $, and so $\left|V(\G) \right| \in \left\lbrace 9, 10, 11 \right\rbrace $. However, since $\G$ is regular with $k=3$, we have $\left|V(\G) \right|=10$ and $\left|D^2_2 \right| =2$.  In \cite{BCCS}, it is shown that the number of connected $3$-regular graphs with $10$ vertices is $19$, but only 5 of them have diameter $d=3$ and girth $g \ge 4$. Out of these five graphs, only four of them have all vertices with eccentricity 3, see Figure \ref{G1}. It is easy to see that none of these graphs is NDB with $\gamma=4$. This shows that $D^2_2=\emptyset$, and so $|V(\G)|=8$. But it is well-known (and also easy to see) that $Q_3$ is the only cubic triangle-free graph with eight vertices and diameter $d=3$.

\end{proof}

{\small\begin{figure}[t]{\rm
\begin{center}
\begin{tikzpicture}[scale=.45]
\draw [line width=1pt, draw=ForestGreen] (3.996232183331462,-1.3232284514950061)-- (4.003182934331805,1.266193600374427);
\draw [line width=1pt, draw=ForestGreen] (4.003182934331805,1.266193600374427)-- (2.486782115964983,3.3651655948764025);
\draw [line width=1pt, draw=ForestGreen] (2.486782115964983,3.3651655948764025)-- (0.02624330027896704,4.171951571545325);
\draw [line width=1pt, draw=ForestGreen] (0.02624330027896704,4.171951571545325)-- (-2.438591315772598,3.378386708940445);
\draw [line width=1pt, draw=ForestGreen] (-2.438591315772598,3.378386708940445)-- (-3.9662386855053127,1.2875858122991861);
\draw [line width=1pt, draw=ForestGreen] (-3.9662386855053127,1.2875858122991861)-- (-3.9731894365056553,-1.301836239570246);
\draw [line width=1pt, draw=ForestGreen] (-3.9731894365056553,-1.301836239570246)-- (-2.456788618138834,-3.4008082340722225);
\draw [line width=1pt, draw=ForestGreen] (-2.456788618138834,-3.4008082340722225)-- (0.003750197547181866,-4.207594210741146);
\draw [line width=1pt, draw=ForestGreen] (0.003750197547181866,-4.207594210741146)-- (2.4685848135987474,-3.4140293481362654);
\draw [line width=1pt, draw=ForestGreen] (2.4685848135987474,-3.4140293481362654)-- (3.996232183331462,-1.3232284514950061);
\draw [line width=1pt, draw=ForestGreen] (0.003750197547181866,-4.207594210741146)-- (2.4685848135987474,-3.4140293481362654);
\draw [line width=1pt, draw=ForestGreen] (2.4685848135987474,-3.4140293481362654)-- (3.996232183331462,-1.3232284514950061);
\draw [line width=1pt, draw=ForestGreen] (3.996232183331462,-1.3232284514950061)-- (4.003182934331805,1.266193600374427);
\draw [line width=1pt, draw=ForestGreen] (4.003182934331805,1.266193600374427)-- (2.486782115964983,3.3651655948764025);
\draw [line width=1pt, draw=ForestGreen] (2.486782115964983,3.3651655948764025)-- (0.02624330027896704,4.171951571545325);
\draw [line width=1pt, draw=ForestGreen] (0.02624330027896704,4.171951571545325)-- (-2.438591315772598,3.378386708940445);
\draw [line width=1pt, draw=ForestGreen] (-2.438591315772598,3.378386708940445)-- (-3.9662386855053127,1.2875858122991861);
\draw [line width=1pt, draw=ForestGreen] (-3.9662386855053127,1.2875858122991861)-- (-3.9731894365056553,-1.301836239570246);
\draw [line width=1pt, draw=ForestGreen] (-3.9731894365056553,-1.301836239570246)-- (-2.456788618138834,-3.4008082340722225);
\draw [line width=1pt, draw=ForestGreen] (0.02624330027896704,4.171951571545325)-- (0.003750197547181866,-4.207594210741146);
\draw [line width=1pt, draw=ForestGreen] (-2.438591315772598,3.378386708940445)-- (4.003182934331805,1.266193600374427);
\draw [line width=1pt, draw=ForestGreen] (2.486782115964983,3.3651655948764025)-- (-3.9662386855053127,1.2875858122991861);
\draw [line width=1pt, draw=ForestGreen] (-2.456788618138834,-3.4008082340722225)-- (3.996232183331462,-1.3232284514950061);
\draw [line width=1pt, draw=ForestGreen] (2.4685848135987474,-3.4140293481362654)-- (-3.9731894365056553,-1.301836239570246);
\fill[color=ForestGreenTwo] (3.996232183331462,-1.3232284514950061) circle [radius=0.23];
\fill[color=ForestGreenTwo] (4.003182934331805,1.266193600374427) circle [radius=0.23];
\fill[color=ForestGreenTwo] (2.486782115964983,3.3651655948764025) circle [radius=0.23];
\fill[color=ForestGreenTwo] (0.02624330027896704,4.171951571545325) circle [radius=0.23];
\fill[color=ForestGreenTwo] (-2.438591315772598,3.378386708940445) circle [radius=0.23];
\fill[color=ForestGreenTwo] (-3.9662386855053127,1.2875858122991861) circle [radius=0.23];
\fill[color=ForestGreenTwo] (-3.9731894365056553,-1.301836239570246) circle [radius=0.23];
\fill[color=ForestGreenTwo] (-2.456788618138834,-3.4008082340722225) circle [radius=0.23];
\fill[color=ForestGreenTwo] (0.003750197547181866,-4.207594210741146) circle [radius=0.23];
\fill[color=ForestGreenTwo] (2.4685848135987474,-3.4140293481362654) circle [radius=0.23];
\end{tikzpicture}
\begin{tikzpicture}[scale=.45]
\draw [line width=1pt, draw=ForestGreen] (3.996232183331462,-1.3232284514950061)-- (4.003182934331805,1.266193600374427);
\draw [line width=1pt, draw=ForestGreen] (4.003182934331805,1.266193600374427)-- (2.486782115964983,3.3651655948764025);
\draw [line width=1pt, draw=ForestGreen] (2.486782115964983,3.3651655948764025)-- (0.02624330027896704,4.171951571545325);
\draw [line width=1pt, draw=ForestGreen] (0.02624330027896704,4.171951571545325)-- (-2.438591315772598,3.378386708940445);
\draw [line width=1pt, draw=ForestGreen] (-2.438591315772598,3.378386708940445)-- (-3.9662386855053127,1.2875858122991861);
\draw [line width=1pt, draw=ForestGreen] (-3.9662386855053127,1.2875858122991861)-- (-3.9731894365056553,-1.301836239570246);
\draw [line width=1pt, draw=ForestGreen] (-3.9731894365056553,-1.301836239570246)-- (-2.456788618138834,-3.4008082340722225);
\draw [line width=1pt, draw=ForestGreen] (-2.456788618138834,-3.4008082340722225)-- (0.003750197547181866,-4.207594210741146);
\draw [line width=1pt, draw=ForestGreen] (0.003750197547181866,-4.207594210741146)-- (2.4685848135987474,-3.4140293481362654);
\draw [line width=1pt, draw=ForestGreen] (2.4685848135987474,-3.4140293481362654)-- (3.996232183331462,-1.3232284514950061);
\draw [line width=1pt, draw=ForestGreen] (0.003750197547181866,-4.207594210741146)-- (2.4685848135987474,-3.4140293481362654);
\draw [line width=1pt, draw=ForestGreen] (2.4685848135987474,-3.4140293481362654)-- (3.996232183331462,-1.3232284514950061);
\draw [line width=1pt, draw=ForestGreen] (3.996232183331462,-1.3232284514950061)-- (4.003182934331805,1.266193600374427);
\draw [line width=1pt, draw=ForestGreen] (4.003182934331805,1.266193600374427)-- (2.486782115964983,3.3651655948764025);
\draw [line width=1pt, draw=ForestGreen] (2.486782115964983,3.3651655948764025)-- (0.02624330027896704,4.171951571545325);
\draw [line width=1pt, draw=ForestGreen] (0.02624330027896704,4.171951571545325)-- (-2.438591315772598,3.378386708940445);
\draw [line width=1pt, draw=ForestGreen] (-2.438591315772598,3.378386708940445)-- (-3.9662386855053127,1.2875858122991861);
\draw [line width=1pt, draw=ForestGreen] (-3.9662386855053127,1.2875858122991861)-- (-3.9731894365056553,-1.301836239570246);
\draw [line width=1pt, draw=ForestGreen] (-3.9731894365056553,-1.301836239570246)-- (-2.456788618138834,-3.4008082340722225);
\draw [line width=1pt, draw=ForestGreen] (-3.9731894365056553,-1.301836239570246)-- (3.996232183331462,-1.3232284514950061);
\draw [line width=1pt, draw=ForestGreen] (-3.9662386855053127,1.2875858122991861)-- (4.003182934331805,1.266193600374427);
\draw [line width=1pt, draw=ForestGreen] (-2.456788618138834,-3.4008082340722225)-- (-2.438591315772598,3.378386708940445);
\draw [line width=1pt, draw=ForestGreen] (0.003750197547181866,-4.207594210741146)-- (0.02624330027896704,4.171951571545325);
\draw [line width=1pt, draw=ForestGreen] (2.4685848135987474,-3.4140293481362654)-- (2.486782115964983,3.3651655948764025);
\fill[color=ForestGreenTwo] (-3.9731894365056553,-1.301836239570246) circle [radius=0.23];
\fill[color=ForestGreenTwo] (-2.456788618138834,-3.4008082340722225) circle [radius=0.23];
\fill[color=ForestGreenTwo] (0.003750197547181866,-4.207594210741146) circle [radius=0.23];
\fill[color=ForestGreenTwo] (2.4685848135987474,-3.4140293481362654) circle [radius=0.23];
\fill[color=ForestGreenTwo] (3.996232183331462,-1.3232284514950061) circle [radius=0.23];
\fill[color=ForestGreenTwo] (4.003182934331805,1.266193600374427) circle [radius=0.23];
vrh 7
\fill[color=ForestGreenTwo] (2.486782115964983,3.3651655948764025) circle [radius=0.23];
\fill[color=ForestGreenTwo] (0.02624330027896704,4.171951571545325) circle [radius=0.23];
\fill[color=ForestGreenTwo] (-2.438591315772598,3.378386708940445) circle [radius=0.23];
\fill[color=ForestGreenTwo] (-3.9662386855053127,1.2875858122991861) circle [radius=0.23];
\end{tikzpicture}
\begin{tikzpicture}[scale=.45]
\draw [line width=1pt, draw=ForestGreen] (3.996232183331462,-1.3232284514950061)-- (4.003182934331805,1.266193600374427);
\draw [line width=1pt, draw=ForestGreen] (4.003182934331805,1.266193600374427)-- (2.486782115964983,3.3651655948764025);
\draw [line width=1pt, draw=ForestGreen] (2.486782115964983,3.3651655948764025)-- (0.02624330027896704,4.171951571545325);
\draw [line width=1pt, draw=ForestGreen] (0.02624330027896704,4.171951571545325)-- (-2.438591315772598,3.378386708940445);
\draw [line width=1pt, draw=ForestGreen] (-2.438591315772598,3.378386708940445)-- (-3.9662386855053127,1.2875858122991861);
\draw [line width=1pt, draw=ForestGreen] (-3.9662386855053127,1.2875858122991861)-- (-3.9731894365056553,-1.301836239570246);
\draw [line width=1pt, draw=ForestGreen] (-3.9731894365056553,-1.301836239570246)-- (-2.456788618138834,-3.4008082340722225);
\draw [line width=1pt, draw=ForestGreen] (-2.456788618138834,-3.4008082340722225)-- (0.003750197547181866,-4.207594210741146);
\draw [line width=1pt, draw=ForestGreen] (0.003750197547181866,-4.207594210741146)-- (2.4685848135987474,-3.4140293481362654);
\draw [line width=1pt, draw=ForestGreen] (2.4685848135987474,-3.4140293481362654)-- (3.996232183331462,-1.3232284514950061);
\draw [line width=1pt, draw=ForestGreen] (0.003750197547181866,-4.207594210741146)-- (2.4685848135987474,-3.4140293481362654);
\draw [line width=1pt, draw=ForestGreen] (2.4685848135987474,-3.4140293481362654)-- (3.996232183331462,-1.3232284514950061);
\draw [line width=1pt, draw=ForestGreen] (3.996232183331462,-1.3232284514950061)-- (4.003182934331805,1.266193600374427);
\draw [line width=1pt, draw=ForestGreen] (4.003182934331805,1.266193600374427)-- (2.486782115964983,3.3651655948764025);
\draw [line width=1pt, draw=ForestGreen] (2.486782115964983,3.3651655948764025)-- (0.02624330027896704,4.171951571545325);
\draw [line width=1pt, draw=ForestGreen] (0.02624330027896704,4.171951571545325)-- (-2.438591315772598,3.378386708940445);
\draw [line width=1pt, draw=ForestGreen] (-2.438591315772598,3.378386708940445)-- (-3.9662386855053127,1.2875858122991861);
\draw [line width=1pt, draw=ForestGreen] (-3.9662386855053127,1.2875858122991861)-- (-3.9731894365056553,-1.301836239570246);
\draw [line width=1pt, draw=ForestGreen] (-3.9731894365056553,-1.301836239570246)-- (-2.456788618138834,-3.4008082340722225);
\draw [line width=1pt, draw=ForestGreen] (0.02624330027896704,4.171951571545325)-- (0.003750197547181866,-4.207594210741146);
\draw [line width=1pt, draw=ForestGreen] (2.486782115964983,3.3651655948764025)-- (-2.456788618138834,-3.4008082340722225);
\draw [line width=1pt, draw=ForestGreen] (-2.438591315772598,3.378386708940445)-- (2.4685848135987474,-3.4140293481362654);
\draw [line width=1pt, draw=ForestGreen] (-3.9662386855053127,1.2875858122991861)-- (4.003182934331805,1.266193600374427);
\draw [line width=1pt, draw=ForestGreen] (-3.9731894365056553,-1.301836239570246)-- (3.996232183331462,-1.3232284514950061);
\fill[color=ForestGreenTwo] (3.996232183331462,-1.3232284514950061) circle [radius=0.23];
\fill[color=ForestGreenTwo] (4.003182934331805,1.266193600374427) circle [radius=0.23];
\fill[color=ForestGreenTwo] (2.486782115964983,3.3651655948764025) circle [radius=0.23];
\fill[color=ForestGreenTwo] (0.02624330027896704,4.171951571545325) circle [radius=0.23];
\fill[color=ForestGreenTwo] (-2.438591315772598,3.378386708940445) circle [radius=0.23];
\fill[color=ForestGreenTwo] (-3.9662386855053127,1.2875858122991861) circle [radius=0.23];
\fill[color=ForestGreenTwo] (-3.9731894365056553,-1.301836239570246) circle [radius=0.23];
\fill[color=ForestGreenTwo] (-2.456788618138834,-3.4008082340722225) circle [radius=0.23];
\fill[color=ForestGreenTwo] (0.003750197547181866,-4.207594210741146) circle [radius=0.23];
\fill[color=ForestGreenTwo] (2.4685848135987474,-3.4140293481362654) circle [radius=0.23];
\end{tikzpicture}
\begin{tikzpicture}[scale=.45]
\draw [line width=1pt, draw=ForestGreen] (3.996232183331462,-1.3232284514950061)-- (4.003182934331805,1.266193600374427);
\draw [line width=1pt, draw=ForestGreen] (4.003182934331805,1.266193600374427)-- (2.486782115964983,3.3651655948764025);
\draw [line width=1pt, draw=ForestGreen] (2.486782115964983,3.3651655948764025)-- (0.02624330027896704,4.171951571545325);
\draw [line width=1pt, draw=ForestGreen] (0.02624330027896704,4.171951571545325)-- (-2.438591315772598,3.378386708940445);
\draw [line width=1pt, draw=ForestGreen] (-2.438591315772598,3.378386708940445)-- (-3.9662386855053127,1.2875858122991861);
\draw [line width=1pt, draw=ForestGreen] (-3.9662386855053127,1.2875858122991861)-- (-3.9731894365056553,-1.301836239570246);
\draw [line width=1pt, draw=ForestGreen] (-3.9731894365056553,-1.301836239570246)-- (-2.456788618138834,-3.4008082340722225);
\draw [line width=1pt, draw=ForestGreen] (-2.456788618138834,-3.4008082340722225)-- (0.003750197547181866,-4.207594210741146);
\draw [line width=1pt, draw=ForestGreen] (0.003750197547181866,-4.207594210741146)-- (2.4685848135987474,-3.4140293481362654);
\draw [line width=1pt, draw=ForestGreen] (2.4685848135987474,-3.4140293481362654)-- (3.996232183331462,-1.3232284514950061);
\draw [line width=1pt, draw=ForestGreen] (0.003750197547181866,-4.207594210741146)-- (2.4685848135987474,-3.4140293481362654);
\draw [line width=1pt, draw=ForestGreen] (2.4685848135987474,-3.4140293481362654)-- (3.996232183331462,-1.3232284514950061);
\draw [line width=1pt, draw=ForestGreen] (3.996232183331462,-1.3232284514950061)-- (4.003182934331805,1.266193600374427);
\draw [line width=1pt, draw=ForestGreen] (4.003182934331805,1.266193600374427)-- (2.486782115964983,3.3651655948764025);
\draw [line width=1pt, draw=ForestGreen] (2.486782115964983,3.3651655948764025)-- (0.02624330027896704,4.171951571545325);
\draw [line width=1pt, draw=ForestGreen] (0.02624330027896704,4.171951571545325)-- (-2.438591315772598,3.378386708940445);
\draw [line width=1pt, draw=ForestGreen] (-2.438591315772598,3.378386708940445)-- (-3.9662386855053127,1.2875858122991861);
\draw [line width=1pt, draw=ForestGreen] (-3.9662386855053127,1.2875858122991861)-- (-3.9731894365056553,-1.301836239570246);
\draw [line width=1pt, draw=ForestGreen] (-3.9731894365056553,-1.301836239570246)-- (-2.456788618138834,-3.4008082340722225);
\draw [line width=1pt, draw=ForestGreen] (0.02624330027896704,4.171951571545325)-- (0.003750197547181866,-4.207594210741146);
\draw [line width=1pt, draw=ForestGreen] (-2.438591315772598,3.378386708940445)-- (2.4685848135987474,-3.4140293481362654);
\draw [line width=1pt, draw=ForestGreen] (-2.456788618138834,-3.4008082340722225)-- (2.486782115964983,3.3651655948764025);
\draw [line width=1pt, draw=ForestGreen] (-3.9662386855053127,1.2875858122991861)-- (3.996232183331462,-1.3232284514950061);
\draw [line width=1pt, draw=ForestGreen] (4.003182934331805,1.266193600374427)-- (-3.9731894365056553,-1.301836239570246);
\fill[color=ForestGreenTwo] (3.996232183331462,-1.3232284514950061) circle [radius=0.23];
\fill[color=ForestGreenTwo] (4.003182934331805,1.266193600374427) circle [radius=0.23];
\fill[color=ForestGreenTwo] (2.486782115964983,3.3651655948764025) circle [radius=0.23];
\fill[color=ForestGreenTwo] (0.02624330027896704,4.171951571545325) circle [radius=0.23];
\fill[color=ForestGreenTwo] (-2.438591315772598,3.378386708940445) circle [radius=0.23];
\fill[color=ForestGreenTwo] (-3.9662386855053127,1.2875858122991861) circle [radius=0.23];
\fill[color=ForestGreenTwo] (-3.9731894365056553,-1.301836239570246) circle [radius=0.23];
\fill[color=ForestGreenTwo] (-2.456788618138834,-3.4008082340722225) circle [radius=0.23];
\fill[color=ForestGreenTwo] (0.003750197547181866,-4.207594210741146) circle [radius=0.23];
\fill[color=ForestGreenTwo] (2.4685848135987474,-3.4140293481362654) circle [radius=0.23];
\end{tikzpicture}
\caption{\rm 
Connected $3$-regular graphs of order $10$ with diameter $d=3$, girth $g\geq 4$ and with all vertices with eccentricity $3$.}
\label{G1}
\end{center}
}\end{figure}
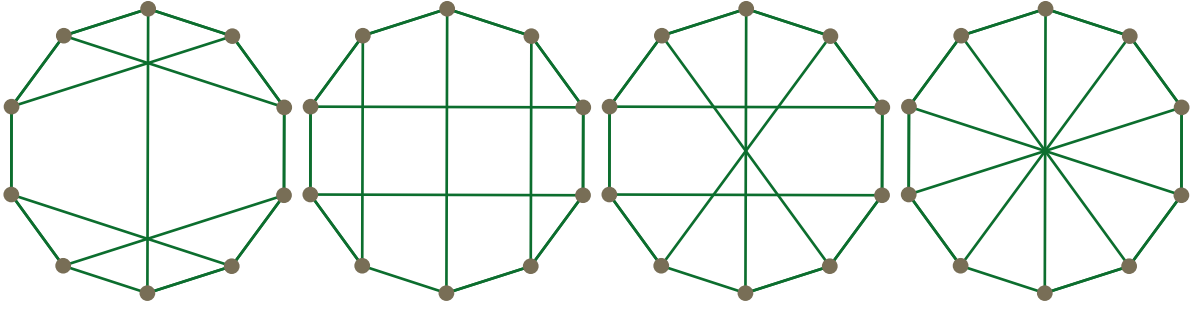}


 \section{Case $k=4$}
 \label{sec:k=4}
 
 Let $\G$ denote a regular NDB graph with valency $k=4$, diameter $d \ge 3$ and $\gamma=\gamma(\G)=d+1$. Recall that by Theorem~\ref{thm:diameter}(ii) we have $d \in \{3,4\}$. In this section we first show that case $d=4$ is not possible, and then classify regular  NDB graphs with $k=4$ and $d=3$. We start with the following lemma.

\begin{lemma}
	\label{lem:k=4d=4_ell_ge_3}
	Let $\G$ denote a regular NDB graph with valency $k=4$, diameter $d = 4$ and $\gamma=\gamma(\G)=d+1$.  Pick vertices $x_0, x_4$ of $\G$ such that $d(x_0,x_4)=4$, and pick a shortest path $x_0, x_1, x_2, x_3, x_4$ between $x_0$ and $x_4$. Let $\ell=\ell(x_0,x_1)$, $D^i_j=D^i_j(x_1, x_0)$ and $D^{\ell-1}_{\ell}=\{x_{\ell},u\}$. Then $\ell=2$. Moreover, $u \sim x_2$ and $u \sim x_3$. 
\end{lemma}

\begin{proof}
Assume first that $\ell=4$. By Proposition~\ref{p1}(i), vertex $x_3$ has a neighbour $z$ in $D_3^3$. Now $W_{x_2, x_1} = \{x_2, x_3, x_4,u,z\}$, and so $x_2$ has no neighbours in $D^2_2\cup D^2_1$. Consequently, $x_2$ has two neighbours in $D^1_1$, contradicting Proposition~\ref{p1}(ii).

Assume now that $\ell=3$. By Proposition~\ref{p1}(i) $x_3$ does not have neighbours in $D^3_2\cup D^2_2$, and so by Proposition~\ref{p2}(ii) we get that $x_3$ and $u$ are adjacent, and that $x_3$ has a neighbour $z$ in $D^3_3$. By Proposition~\ref{p2}(ii) vertex $x_2$ has no neighbours in $D^2_2 \cup D^2_1$, and so $x_2$ has a neighbour $w$ in $D^1_1$. Now $\{x_3,x_2,x_1,x_0,w\} \subseteq W_{x_3, x_4}$, implying that $x_4$ is adjacent to both $u$ and $z$. Similarly,  $\{u,x_2,x_1,x_0,w\} \subseteq W_{u, x_4}$, and so $u$ has no neighbours in $D^2_2 \cup D^3_2$. It follows that $u$ has a neighbour in $D^3_3$, and  by Proposition~\ref{p2}(ii), this neighbour is $z$. But now the edge $x_3u$ is not balanced, a contradiction.

This shows that $\ell=2$. By Proposition~\ref{p1}(i), vertex $x_3$ has either one or two neighbours in $D^3_3$. If $x_3$ has two neighbours in $D^3_3$, then by Proposition~\ref{p2}(i) vertex $x_2$ has no neighbours in $D^2_2\cup D^2_1$. Therefore, $x_2$ is adjacent to the unique vertex $w \in D^1_1$, and is also adjacent to $u$. But now we have that $\{x_3,x_2,x_1,x_0,u,w\} \subseteq W_{x_3, x_4}$, a contradiction. 

Therefore, $x_3$ has exactly one neighbour in $D^3_3$. As by Proposition~\ref{p1}(i) vertex $x_3$ has no neighbours in $D^2_2 \cup D^3_2$, we have that $x_3 \sim u$. Consequently $\{x_3,x_2,x_1,x_0,u\} \subseteq W_{x_3, x_4}$, and so $x_2$ and $u$ have no neighbours in $D^1_1 \cup D^2_1$.  
Since $k=4$ and since edges $x_2x_1$ and $u x_1$ are balanced, it follows  both of  $x_2$ and $u$ have exactly one neighbour in $D^2_2$, and that $x_2 \sim u$.
\end{proof}

 \begin{proposition}
	\label{prop:k=4d=4}
	Let $\G$ denote a regular NDB graph with valency $k=4$, diameter $d \ge 3$ and $\gamma=\gamma(\G)=d+1$.  Then $d \ne 4$. 
\end{proposition}
\begin{proof}
	Assume on the contrary that $d=4$. Pick vertices $x_0, x_4$ of $\G$ such that $d(x_0,x_4)=4$. Pick a shortest path $x_0, x_1, x_2, x_3, x_4$ between $x_0$ and $x_4$. Let $D^i_j = D^i_j(x_1, x_0)$, let $\ell=\ell(x_0,x_1)$ and let $D^{\ell-1}_{\ell} = \{x_{\ell},u\}$. 
    Recall that by Lemma \ref{lem:k=4d=4_ell_ge_3} we have that $\ell=2$ and that vertex $u$ is adjacent with $x_2$ and $x_3$. Let $z$ denote a neighbour of $x_3$ in $D^3_3$ (note that by Proposition~\ref{p1}(i) vertex $x_3$ has no neighbours in $D^2_2 \cup D^3_2$).
    
    Since $W_{x_3, x_4}=\{x_3, x_2, x_1, x_0,u\}$, vertices $x_2$ and $u$ have no neighbours in $D^1_1 \cup D^2_1$. Let us denote the neighbours of $u$ and $x_2$ in $D^2_2$ by $v_1$, $v_2$, respectively. Note that $v_1 \ne v_2$, otherwise edge $ux_2$ is not balanced. Furthermore, $\{x_3,x_2,x_1,x_0,u\} =W_{x_3, x_4}$ implies that $x_4$ and $z$ are adjacent, and that $x_4$ is at distance $2$ from both $v_1$ and $v_2$. Consequently, $v_1$ and $v_2$ both have a common neighbour, say $z_1$ and $z_2$ respectively, with $x_4$, and these common neighbours must be in $D^3_3$. But as edges $x_2 x_1$ and $u x_1$ are balanced, this implies that $z_1 = z = z_2$ (see Figure~\ref{07}(b)).


Note that $v_1$ and $v_2$ both have at least one neighbour  in $D^1_1 \cup D^2_1$. Let us denote a neighbour of $v_1$ ($v_2$, respectively) in $D^1_1 \cup D^2_1$ by $w_1$ ($w_2$, respectively). If $w_1 \ne w_2$, then $\{z, v_1, v_2, w_1, w_2, x_0\} \subseteq W_{z,x_4}$, contradicting $\gamma=5$. Therefore $w_1=w_2$ and by applying Lemma \ref{lem:k=4d=4_ell_ge_3} to the path $x_0, w_1, v_1,z, x_4$ we get that vertices $v_1$ and $v_2$ are adjacent. But now it is easy to see that $W_{u, x_2}=\{u,v_1\}$, a contradiction. This finishes the proof.
\end{proof}

 \begin{proposition}
 	\label{prop:k=4d=3l=3}
  Let $\G$ denote a regular NDB graph with valency $k=4$, diameter $d = 3$ and $\gamma=\gamma(\G)=4$.  Then for every edge $x_0 x_1$ of $\G$ we have that $|D^1_2(x_1, x_0)|=|D^2_1(x_1, x_0)|=2$. 
\end{proposition}
\begin{proof}
	Pick an edge $x_0 x_1$ of $\G$ and let $D^i_j = D^i_j(x_1, x_0)$. By Proposition~\ref{nonempty} we have that $D^2_3 \ne \emptyset$, and so $\gamma=4$ implies $|D^1_2| \le 2$. Assume to the contrary that $|D^1_2|=1$, and so $|D^2_3|=2$, $|D^1_1|=2$ and $|D^2_1|=1$. Let $x_3, u$ be vertices of $D^2_3$, and let $x_2$ be the unique vertex of $D^1_2$. Let $z$ denote the neighbour of $x_2$, different from $x_1, x_3, u$, and note that $z \in D^2_2 \cup D^2_1 \cup D^1_1$. In each of these three cases we derive a contradiction. 
	
	Assume first that $z \in D^2_2$. Then $D^1_2(x_2,x_1)=\{x_3,u,z\}$, and $\gamma=4$ forces $D^2_3(x_2,x_1)=\emptyset$, contradicting Proposition~\ref{nonempty}.
	
	Assume next that $z \in D^2_1$ (note that $z$ is the unique vertex in $D^2_1$). Then $\{x_2, z, x_3, u\} \cup D^3_2 \subseteq W_{x_2 ,x_1}$. As $D^3_2 \ne \emptyset$ by Proposition~\ref{nonempty}, this contradicts $\gamma=4$.
	
	Assume finally that $z \in D^1_1$. Recall that $|D^1_1|=2$ and denote the other vertex of $D^1_1$ by $w$. If $z$ and $w$ are adjacent, then $W_{x_1, z}=\{x_1\}$, a contradiction. If $z$ has a neighbour $v \in D^2_2$, then $\{z,v,x_2,u,x_3\} \subseteq W_{z, x_0}$, a contradiction. This shows that $z$ is adjacent to the unique vertex of $D^2_1$. Let us denote this vertex by $y_2$. As $W_{x_2, x_3}=W_{x_2, u} =\{x_2, x_1, x_0, z\}$, vertices $x_3$ and $u$ are both at distance $2$ from $y_2$. But this shows that $W_{z, y_2}=\{x_1, z, x_2\}$, a contradiction.
\end{proof}

\begin{theorem}
	\label{k44mm}
	Let $\G$ denote a regular NDB graph with valency $k=4$, diameter $d \ge 3$ and $\gamma=\gamma(\G)=d+1$.  Then $\G$ is isomorphic to the line graph of the $3$-dimensional hypercube $Q_3$.
\end{theorem}
\begin{proof}
	By Theorem~\ref{thm:diameter}(ii) and Proposition~\ref{prop:k=4d=4} we have that $d =3$. Pick an arbitrary edge $xy$ of $\G$. By Proposition~\ref{prop:k=4d=3l=3} we have that $|D^1_2(x,y)|=|D^2_1(x,y)|=2$. Consequently $|D^1_1(x,y)|=1$, and so $\G$ is an edge-regular graph with $\lambda=1$. Observe that $\gamma=4$ also implies that $|D^2_3(x,y)|=|D^3_2(x,y)|=1$. Observe that $\G$ contains $|V(\G)| k /6 = 2|V(\G)|/3$ triangles, and so $|V(\G)|$ is divisible by $3$.
	
Pick vertices $x_0, x_3$ of $\G$ at distance $3$ and let $x_0,x_1,x_2,x_3$ be a shortest path from $x_0$ to $x_3$.  Abbreviate $D^i_j=D^i_j(x_1,x_0)$. Obviously $D^2_3=\{x_3\}$ and $x_2 \in D^1_2$. Let us denote the other vertex of $D^1_2$ by $u$, the vertices of $D^2_1$ by $y_2, v$, the vertex of $D^3_2$ by $y_3$ and the vertex of $D^1_1$ by $w$. Without loss of generality we may assume that $y_2$ and $y_3$ are adjacent. Since $\G$ is edge-regular with $\lambda=1$, we also obtain that $x_2$ and $u$ are adjacent, that $y_2$ and $v$ are adjacent, and that $w$ has two neighbours, say $z_1$ and $z_2$, in $D^2_2$, and that $z_1, z_2$ are also adjacent. As $W_{x_2, x_3}=\{x_2,x_1,x_0,u\}$, $x_3$ is at distance $2$ from $w$, and so $x_3$ is adjacent to exactly one of $z_1, z_2$, Without loss of generality we could assume that $x_3$ and $z_1$ are adjacent. 
	
Note that $\G(w)=\{x_0,x_1,z_1,z_2\}$, and so $x_2$ and $w$ are not adjacent. {Vertex $x_2$ is also not adjacent to $y_2$, as otherwise edge $x_2 y_2$ is not contained in a triangle. If $x_2\sim v$ then $v\sim u$ and the edge $ux_2$ is contained in two triangles, contradicting $\lambda=1$. It follows that $x_2$ has no neighbours in $D^2_1$. Therefore, $x_2$ has a neighbour in $D^2_2$. }Consequently, by Proposition~\ref{p2}(i), $x_3$ could have at most one neighbour in $D^3_3 \cup D^3_2$. 

We now show that $D_3^3 = \emptyset$. Assume to the contrary that there exists $t \in D^3_3$. If $t$ is adjacent to $z_1$ or $z_2$, then $\{w,z_1,z_2,x_3,t\} \subseteq W_{w, x_0}$, a contradiction. If $t$ is adjacent with $z \in D^2_2 \setminus \{z_1, z_2\}$, then $z$ has a neighbour in $D^1_2$ and a neighbour in $D^2_1$, implying that $|W_{z, t}| \ge 5$, a contradiction. It follows that $t$ has no neighbours in $D^2_2$, and so $t$ is adjacent with $x_3$ (and with $y_3)$. Now the unique common neighbour of $x_3$ and $t$ must be contained in $D^3_3 \cup D^3_2$, contradicting the fact that $x_3$ could have at most one neighbour in $D^3_3 \cup D^3_2$. This shows that $D^3_3 = \emptyset$. 
	
	Let us now estimate the cardinality of $D^2_2$. Observe that each $z  \in D^2_2 \setminus \{z_1, z_2\}$ has a neighbour in $D^1_2$. But $u$ could have at most two neighbours in $D^2_2$, while $x_2$ has exactly one neighbour in $D^2_2$. It follows that $2 \le |D^2_2| \le 5$, and so $11 \le |V(\G)| \le 14$. As $|V(\G)|$ is divisible by 3, we have that $|V(\G)|=12$. By \cite[Corollary 6]{GHJR}, there are just two edge-regular graphs on 12 vertices with $\lambda=1$, namely the line graph of $3$-dimensional hypercube (see Figure~\ref{09}), and the line graph of the M\"obius ladder graph on eight vertices. It is easy to see that the latter one is not even distance-balanced.
\end{proof}

{\small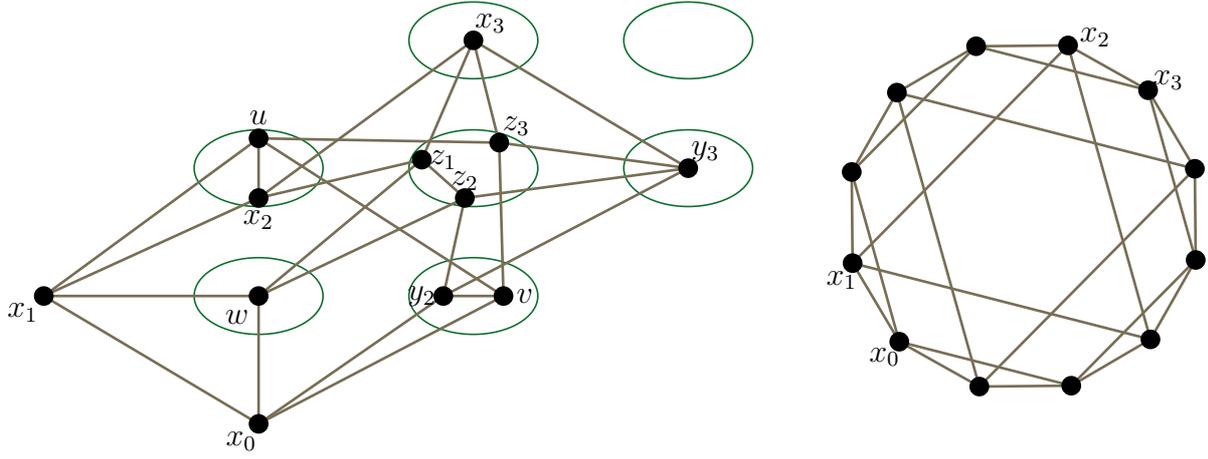
\begin{figure}[t]{\rm
\begin{center}
\begin{tikzpicture}[scale=.565]
\draw[fill=white, draw=ForestGreen, line width=0.6pt] (8.,-1.) ellipse (1.5cm and .9cm);
\draw[fill=white, draw=ForestGreen, line width=0.6pt] (13.,-1.) ellipse (1.5cm and .9cm);
\draw[fill=white, draw=ForestGreen, line width=0.6pt] (8.,-4.) ellipse (1.5cm and .9cm);
\draw[fill=white, draw=ForestGreen, line width=0.6pt] (13.,-4.) ellipse (1.5cm and .9cm);
\draw[fill=white, draw=ForestGreen, line width=0.6pt] (8.,-7.) ellipse (1.5cm and .9cm);
\draw[fill=white, draw=ForestGreen, line width=0.6pt] (3.,-7.) ellipse (1.5cm and .9cm);
\draw[fill=white, draw=ForestGreen, line width=0.6pt] (3.,-4.) ellipse (1.5cm and .9cm);


\draw [line width=1pt, draw=ForestGreenTwo] (8,-1)-- (13,-4);
\draw [line width=1pt, draw=ForestGreenTwo] (-2,-7)-- (3,-10);
\draw [line width=1pt, draw=ForestGreenTwo] (-2,-7)-- (3,-7);
\draw [line width=1pt, draw=ForestGreenTwo] (3,-7)-- (3,-10);
\draw [line width=1pt, draw=ForestGreenTwo] (-2,-7)-- (3,-3.3); 
\draw [line width=1pt, draw=ForestGreenTwo] (-2,-7)-- (3,-4.7); 
\draw [line width=1pt, draw=ForestGreenTwo] (3,-3.3)-- (3,-4.7); 
\draw [line width=1pt, draw=ForestGreenTwo] (8,-1)-- (3,-4.7); 
\draw [line width=1pt, draw=ForestGreenTwo] (3,-10)-- (8.7,-7); 
\draw [line width=1pt, draw=ForestGreenTwo] (3,-10)-- (7.3,-7); 
\draw [line width=1pt, draw=ForestGreenTwo] (8.7,-7)-- (7.3,-7); 
\draw [line width=1pt, draw=ForestGreenTwo] (13,-4)-- (7.3,-7); 
\draw [line width=1pt, draw=ForestGreenTwo] (8.6,-3.4)-- (3,-3.3); 
\draw [line width=1pt, draw=ForestGreenTwo] (8.6,-3.4)-- (8.7,-7); 
\draw [line width=1pt, draw=ForestGreenTwo] (8.6,-3.4)-- (8,-1); 
\draw [line width=1pt, draw=ForestGreenTwo] (8.6,-3.4)-- (13,-4); 
\draw [line width=1pt, draw=ForestGreenTwo] (3,-3.3)-- (8.7,-7); 
\draw [line width=1pt, draw=ForestGreenTwo] (3,-7)-- (6.8,-3.8); 
\draw [line width=1pt, draw=ForestGreenTwo] (3,-7)-- (7.8,-4.7); 
\draw [line width=1pt, draw=ForestGreenTwo] (6.8,-3.8)-- (7.8,-4.7); 
\draw [line width=1pt, draw=ForestGreenTwo] (3,-4.7)-- (6.8,-3.8); 
\draw [line width=1pt, draw=ForestGreenTwo] (8,-1)-- (6.8,-3.8); 
\draw [line width=1pt, draw=ForestGreenTwo] (7.3,-7)-- (7.8,-4.7); 
\draw [line width=1pt, draw=ForestGreenTwo] (13,-4)-- (7.8,-4.7); 

\fill (3,-7) circle [radius=0.23]; 
\fill (8,-1) circle [radius=0.23]; 
\fill (13,-4) circle [radius=0.23]; 
\fill (-2,-7) circle [radius=0.23]; 
\fill (3,-10) circle [radius=0.23]; 
\node at (-2.1,-7,1) {\normalsize $x_1$};
\node at (2.6,-10.4) {\normalsize $x_0$};
\node at (3,-2.8) {\normalsize $u$};
\node at (9.2,-7) {\normalsize $v$};
\fill (3,-4.7) circle [radius=0.23]; 
\fill (3,-3.3) circle [radius=0.23]; 
\fill (8.7,-7) circle [radius=0.23]; 
\fill (7.3,-7) circle [radius=0.23]; 
\fill (8.6,-3.4) circle [radius=0.23]; 
\fill (6.8,-3.8) circle [radius=0.23]; 
\fill (7.8,-4.7) circle [radius=0.23]; 

\node at (2.5,-7.5) {\normalsize $w$}; 
\node at (8.4,-.6) {\normalsize $x_3$}; 
\node at (13.4,-3.6) {\normalsize $y_3$};
\node at (6.8,-7) {\normalsize $y_2$};
\node at (3,-5.2) {\normalsize $x_2$};
\node at (9,-3) {\normalsize $z_3$};
\node at (7.3,-3.8) {\normalsize $z_1$};
\node at (7.8,-4.3) {\normalsize $z_2$};
\end{tikzpicture}\qquad
\begin{tikzpicture}[scale=.565]
\draw[line width=1pt, draw=ForestGreenTwo] (-1.0746410161513757,3.986588728197397)--(2.928653347947322,2.953909236273084);
\draw[line width=1pt, draw=ForestGreenTwo] (-2.917935380250075,2.8992682201217095)--(4.01597385602301,1.1106148721743858);
\draw[line width=1pt, draw=ForestGreenTwo] (-3.950614872174388,-1.1040261439769874)--(2.9832943640986977,-2.8926794919243117);
\draw[line width=1pt, draw=ForestGreenTwo] (-2.8632943640987003,-2.9473205080756872)--(1.14,-3.98);
\draw[line width=1pt, draw=ForestGreenTwo] (-1.,-4.)--(-2.917935380250075,2.8992682201217095);
\draw[line width=1pt, draw=ForestGreenTwo] (-1.,-4.)--(4.01597385602301,1.1106148721743858);
\draw[line width=1pt, draw=ForestGreenTwo] (1.0653589838486228,4.006588728197396)--(-3.950614872174388,-1.1040261439769874);
\draw[line width=1pt, draw=ForestGreenTwo] (1.0653589838486228,4.006588728197396)--(2.9832943640986977,-2.8926794919243117);
\draw[line width=1pt, draw=ForestGreenTwo] (4.035973856023009,-1.029385127825613)--(2.928653347947322,2.953909236273084);
\draw[line width=1pt, draw=ForestGreenTwo] (4.035973856023009,-1.029385127825613)--(1.14,-3.98);
\draw[line width=1pt, draw=ForestGreenTwo] (-3.9706148721743872,1.03597385602301)--(-1.0746410161513757,3.986588728197397);
\draw[line width=1pt, draw=ForestGreenTwo] (-3.9706148721743872,1.03597385602301)--(-2.8632943640987003,-2.9473205080756872);
\draw[line width=1pt, draw=ForestGreenTwo] (-1.,-4.)--(1.14,-3.98);
\draw[line width=1pt, draw=ForestGreenTwo] (1.14,-3.98)--(2.9832943640986977,-2.8926794919243117);
\draw[line width=1pt, draw=ForestGreenTwo] (2.9832943640986977,-2.8926794919243117)--(4.035973856023009,-1.029385127825613);
\draw[line width=1pt, draw=ForestGreenTwo] (4.035973856023009,-1.029385127825613)--(4.01597385602301,1.1106148721743858);
\draw[line width=1pt, draw=ForestGreenTwo] (4.01597385602301,1.1106148721743858)--(2.928653347947322,2.953909236273084);
\draw[line width=1pt, draw=ForestGreenTwo] (2.928653347947322,2.953909236273084)--(1.0653589838486228,4.006588728197396);
\draw[line width=1pt, draw=ForestGreenTwo] (1.0653589838486228,4.006588728197396)--(-1.0746410161513757,3.986588728197397);
\draw[line width=1pt, draw=ForestGreenTwo] (-1.0746410161513757,3.986588728197397)--(-2.917935380250075,2.8992682201217095);
\draw[line width=1pt, draw=ForestGreenTwo] (-2.917935380250075,2.8992682201217095)--(-3.9706148721743872,1.03597385602301);
\draw[line width=1pt, draw=ForestGreenTwo] (-3.9706148721743872,1.03597385602301)--(-3.950614872174388,-1.1040261439769874);
\draw[line width=1pt, draw=ForestGreenTwo] (-3.950614872174388,-1.1040261439769874)--(-2.8632943640987003,-2.9473205080756872);
\draw[line width=1pt, draw=ForestGreenTwo] (-2.8632943640987003,-2.9473205080756872)--(-1.,-4.);
\fill (-1.,-4.) circle [radius=0.23];
\fill (1.14,-3.98) circle [radius=0.23];
\fill (2.9832943640986977,-2.8926794919243117) circle [radius=0.23];
\fill (4.035973856023009,-1.029385127825613) circle [radius=0.23];
\fill (4.01597385602301,1.1106148721743858) circle [radius=0.23];
\fill (2.928653347947322,2.953909236273084) circle [radius=0.23];
\fill (1.0653589838486228,4.006588728197396) circle [radius=0.23];
\fill (-1.0746410161513757,3.986588728197397) circle [radius=0.23];
\fill (-2.917935380250075,2.8992682201217095) circle [radius=0.23];
\fill (-3.9706148721743872,1.03597385602301) circle [radius=0.23];
\fill (-3.950614872174388,-1.1040261439769874) circle [radius=0.23];
\fill (-2.8632943640987003,-2.9473205080756872) circle [radius=0.23];
\fill[color=white] (-2,-5.5) circle [radius=0.23];
\node at (-4.2,-1.5) {\normalsize $x_1$}; 
\node at (-3.2,-3.3) {\normalsize $x_0$}; 
\node at (1.7,4.2) {\normalsize $x_2$}; 
\node at (3.4,3.2) {\normalsize $x_3$}; 
\end{tikzpicture}

\caption{\rm 
The line graph of $Q_3$, drawn on two different ways.
}
\label{09}
\end{center}
}\end{figure}}

 
 \section{Case $k=5$}
 \label{sec:k=5}
 
 Let $\G$ denote a regular NDB graph with valency $k=5$, diameter $d \ge 3$ and $\gamma=\gamma(\G)=d+1$. Recall that by Theorem~\ref{thm:diameter} we have $d=3$, and so $\gamma=4$. In this section we classify such  NDB graphs. We first show that in this case we have $|D^1_2(x_1,x_0)| = |D^2_1(x_1,x_0)|=2$ for every edge $x_1 x_0$ of $\G$.

 \begin{proposition}
 	\label{k=5:prop1}
 	Let $\G$ denote a regular NDB graph with valency $k=5$, diameter $d=3$ and $\gamma=4$. Then for every edge $x_0 x_1$ of $\G$ we have that $|D^1_2(x_1, x_0)|=|D^2_1(x_1, x_0)|=2$. 
 \end{proposition}

\begin{proof}
Pick an edge $x_0 x_1$ of $\G$ and let $D^i_j = D^i_j(x_1, x_0)$. By Proposition~\ref{nonempty} we have that $D^2_3 \ne \emptyset$, and so $\gamma=4$ implies $|D^1_2| \le 2$. Assume to the contrary that $|D^1_2|=1$, and so $|D^2_3|=2$, $|D^1_1|=3$ and $|D^2_1|=1$. Let $x_3, u$ be vertices of $D^2_3$, and let $x_2$ be the unique vertex of $D^1_2$. Let us denote the unique vertex of $D^2_1$ by $y_2$, and the vertices of $D^1_1$ by $z_1, z_2, z_3$. Note that also $|D^3_2|=2$, and let us denote these two vertices by $y_3,u_1$. Clearly we have that $x_2$ is adjacent to both $x_3$ and $u$, and $y_2$ is adjacent to both $y_3$ and $u_1$, see the diagram on the left side of Figure \ref{G2}.

{\small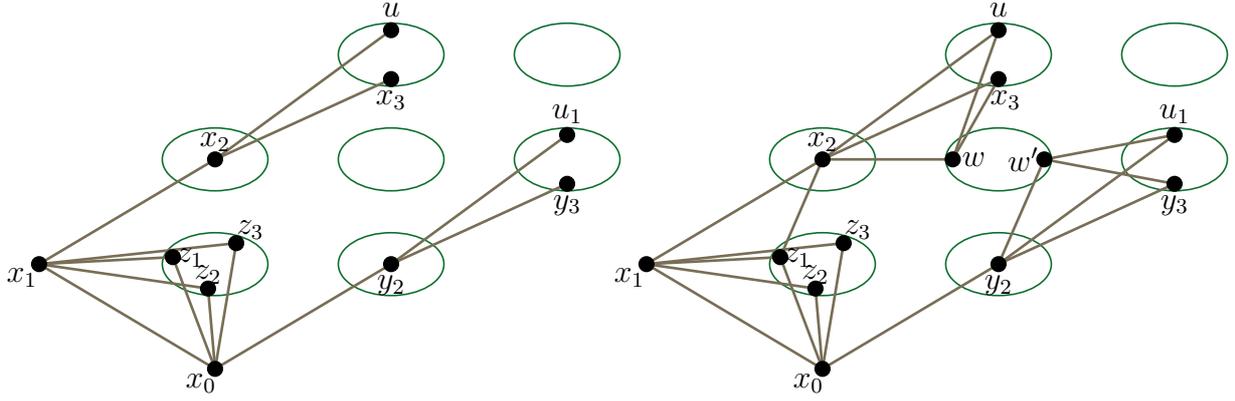
\begin{figure}[t]{\rm
\begin{center}
\begin{tikzpicture}[scale=.463]
\draw[fill=white, draw=ForestGreen, line width=0.6pt] (8.,-1.) ellipse (1.5cm and .9cm);
\draw[fill=white, draw=ForestGreen, line width=0.6pt] (13.,-1.) ellipse (1.5cm and .9cm);
\draw[fill=white, draw=ForestGreen, line width=0.6pt] (8.,-4.) ellipse (1.5cm and .9cm);
\draw[fill=white, draw=ForestGreen, line width=0.6pt] (13.,-4.) ellipse (1.5cm and .9cm);
\draw[fill=white, draw=ForestGreen, line width=0.6pt] (8.,-7.) ellipse (1.5cm and .9cm);
\draw[fill=white, draw=ForestGreen, line width=0.6pt] (3.,-7.) ellipse (1.5cm and .9cm);
\draw[fill=white, draw=ForestGreen, line width=0.6pt] (3.,-4.) ellipse (1.5cm and .9cm);

\draw [line width=1pt, draw=ForestGreenTwo] (3,-10)-- (-2,-7);
\draw [line width=1pt, draw=ForestGreenTwo] (3,-10)-- (1.8,-6.8);
\draw [line width=1pt, draw=ForestGreenTwo] (3,-10)-- (2.8,-7.7);
\draw [line width=1pt, draw=ForestGreenTwo] (3,-10)-- (3.6,-6.4);
\draw [line width=1pt, draw=ForestGreenTwo] (-2,-7)-- (1.8,-6.8);
\draw [line width=1pt, draw=ForestGreenTwo] (-2,-7)-- (2.8,-7.7);
\draw [line width=1pt, draw=ForestGreenTwo] (-2,-7)-- (3.6,-6.4);
\draw [line width=1pt, draw=ForestGreenTwo] (-2,-7)-- (3,-4);
\draw [line width=1pt, draw=ForestGreenTwo] (3,-10)-- (8,-7);
\draw [line width=1pt, draw=ForestGreenTwo] (8,-0.3)-- (3,-4);
\draw [line width=1pt, draw=ForestGreenTwo] (8,-1.7)-- (3,-4);
\draw [line width=1pt, draw=ForestGreenTwo] (8,-7)-- (13,-4.7);
\draw [line width=1pt, draw=ForestGreenTwo] (8,-7)-- (13,-3.3);


\fill (3,-10) circle [radius=0.23]; 
\fill (-2,-7) circle [radius=0.23]; 
\node at (-2.1,-7,1) {\normalsize $x_1$};
\node at (2.6,-10.4) {\normalsize $x_0$};

\fill (3,-4) circle [radius=0.23]; 
\fill (8,-7) circle [radius=0.23]; 
\node at (3,-3.5) {\normalsize $x_2$};
\node at (8,-7.6) {\normalsize $y_2$};

\fill (3.6,-6.4) circle [radius=0.23]; 
\fill (1.8,-6.8) circle [radius=0.23]; 
\fill (2.8,-7.7) circle [radius=0.23]; 
\node at (2.3,-6.8) {\normalsize $z_1$};
\node at (2.8,-7.3) {\normalsize $z_2$};
\node at (4,-6) {\normalsize $z_3$};

\fill (8,-0.3) circle [radius=0.23]; 
\fill (8,-1.7) circle [radius=0.23]; 
\node at (8,0.3) {\normalsize $u$};
\node at (8,-2.3) {\normalsize $x_3$};

\fill (13,-3.3) circle [radius=0.23]; 
\fill (13,-4.7) circle [radius=0.23]; 
\node at (13,-2.7) {\normalsize $u_1$};
\node at (13,-5.3) {\normalsize $y_3$};

\end{tikzpicture}
\hspace{-5mm}
\begin{tikzpicture}[scale=.463]
\draw[fill=white, draw=ForestGreen, line width=0.6pt] (8.,-1.) ellipse (1.5cm and .9cm);
\draw[fill=white, draw=ForestGreen, line width=0.6pt] (13.,-1.) ellipse (1.5cm and .9cm);
\draw[fill=white, draw=ForestGreen, line width=0.6pt] (8.,-4.) ellipse (1.5cm and .9cm);
\draw[fill=white, draw=ForestGreen, line width=0.6pt] (13.,-4.) ellipse (1.5cm and .9cm);
\draw[fill=white, draw=ForestGreen, line width=0.6pt] (8.,-7.) ellipse (1.5cm and .9cm);
\draw[fill=white, draw=ForestGreen, line width=0.6pt] (3.,-7.) ellipse (1.5cm and .9cm);
\draw[fill=white, draw=ForestGreen, line width=0.6pt] (3.,-4.) ellipse (1.5cm and .9cm);

\draw [line width=1pt, draw=ForestGreenTwo] (3,-10)-- (-2,-7);
\draw [line width=1pt, draw=ForestGreenTwo] (3,-10)-- (1.8,-6.8);
\draw [line width=1pt, draw=ForestGreenTwo] (3,-10)-- (2.8,-7.7);
\draw [line width=1pt, draw=ForestGreenTwo] (3,-10)-- (3.6,-6.4);
\draw [line width=1pt, draw=ForestGreenTwo] (-2,-7)-- (1.8,-6.8);
\draw [line width=1pt, draw=ForestGreenTwo] (-2,-7)-- (2.8,-7.7);
\draw [line width=1pt, draw=ForestGreenTwo] (-2,-7)-- (3.6,-6.4);
\draw [line width=1pt, draw=ForestGreenTwo] (-2,-7)-- (3,-4);
\draw [line width=1pt, draw=ForestGreenTwo] (3,-10)-- (8,-7);
\draw [line width=1pt, draw=ForestGreenTwo] (8,-0.3)-- (3,-4);
\draw [line width=1pt, draw=ForestGreenTwo] (8,-1.7)-- (3,-4);
\draw [line width=1pt, draw=ForestGreenTwo] (8,-7)-- (13,-4.7);
\draw [line width=1pt, draw=ForestGreenTwo] (8,-7)-- (13,-3.3);
\draw [line width=1pt, draw=ForestGreenTwo] (6.7,-4)-- (3,-4);
\draw [line width=1pt, draw=ForestGreenTwo] (6.7,-4)-- (8,-1.7);
\draw [line width=1pt, draw=ForestGreenTwo] (6.7,-4)-- (8,-0.3);
\draw [line width=1pt, draw=ForestGreenTwo] (3,-4)-- (1.8,-6.8);
\draw [line width=1pt, draw=ForestGreenTwo] (9.3,-4)-- (8,-7);
\draw [line width=1pt, draw=ForestGreenTwo] (9.3,-4)-- (13,-4.7);
\draw [line width=1pt, draw=ForestGreenTwo] (9.3,-4)-- (13,-3.3);

\fill (6.7,-4) circle [radius=0.23]; 
\fill (9.3,-4) circle [radius=0.23]; 
\node at (7.3,-4) {\normalsize $w$};
\node at (8.7,-4) {\normalsize $w'$};

\fill (3,-10) circle [radius=0.23]; 
\fill (-2,-7) circle [radius=0.23]; 
\node at (-2.1,-7,1) {\normalsize $x_1$};
\node at (2.6,-10.4) {\normalsize $x_0$};
\fill (3,-4) circle [radius=0.23]; 
\fill (8,-7) circle [radius=0.23]; 
\node at (3,-3.5) {\normalsize $x_2$};
\node at (8,-7.6) {\normalsize $y_2$};
\fill (3.6,-6.4) circle [radius=0.23]; 
\fill (1.8,-6.8) circle [radius=0.23]; 
\fill (2.8,-7.7) circle [radius=0.23]; 
\node at (2.3,-6.8) {\normalsize $z_1$};
\node at (2.8,-7.3) {\normalsize $z_2$};
\node at (4,-6) {\normalsize $z_3$};
\fill (8,-0.3) circle [radius=0.23]; 
\fill (8,-1.7) circle [radius=0.23]; 
\node at (8,0.3) {\normalsize $u$};
\node at (8.2,-2.3) {\normalsize $x_3$};
\fill (13,-3.3) circle [radius=0.23]; 
\fill (13,-4.7) circle [radius=0.23]; 
\node at (13,-2.7) {\normalsize $u_1$};
\node at (13,-5.3) {\normalsize $y_3$};

\end{tikzpicture}
\caption{\rm 
Graph $\G$ from Proposition~\ref{k=5:prop1}.
}
\label{G2}
\end{center}
}\end{figure}}

\noindent
Observe that each edge $xy$ of $\G$ is contained in at least one triangle; otherwise $|W_{x,y}| \ge 5 > \gamma$, a contradiction. Therefore, $x_2$ and $y_2$ both have at least one neighbour in $D^1_1$. On the other hand, these two vertices could not have more than one neighbour in $D^1_1$, as otherwise $|W_{x_2, x_3}| \ge 5$ ($|W_{y_2, y_3}| \ge 5$, respectively), a contradiction. Without loss of generality we could assume that $z_1$ is the unique neighbour of $x_2$ in $D^1_1$. Note that it follows from Proposition~\ref{p1}(ii) that $x_2$ and $y_2$ are not adjacent. This shows that $x_2$ has a unique neighbour (say $w$) in $D^2_2$. As $W_{x_2, x_3} = W_{x_2, u}=\{x_2, x_1, x_0,z_1\}$, vertex $w$ is adjacent to both $u$ and $x_3$. Similarly we prove that also $y_2$ has a unique neighbour in $D^2_2$, say $w'$, and that $w'$ is adjacent to both $u_1$ and $y_3$. If $w=w'$, then the degree of $w$ is at least 6, a contradiction. Therefore, $w \ne w'$, see the diagram on the right side of Figure \ref{G2}.

 	{Note that $W_{x_2, x_1} = \{x_2,x_3,u,w\}$, and so both $y_3$ and $u_1$ are at distance 3 from $x_2$. Similarly, $W_{x_1, x_2} = \{x_1,x_0,z_2, z_3\}$, and so $y_2$ is at distance $2$ from $x_2$.  Therefore $y_2$ and $x_2$ have a common neighbour, and by the comments above the only possible common neighbour is $z_1$. It follows that $z_1$ and $y_2$ are adjacent. But now $\{y_2, x_0, x_1, z_1, x_2\} \subseteq W_{y_2,  y_3}$ (recall that $d(x_2,y_3)=3$), a contradiction. This shows that $|D^1_2|=2$. By Lemma \ref{eq} we obtain that $|D^2_1|=2$ as well.}
 \end{proof}

\begin{theorem}
	\label{k=5:thm1}
		Let $\G$ denote a regular NDB graph with valency $k=5$, diameter $d \ge 3$ and $\gamma=d+1$.Then $\G$ is isomorphic to the icosahedron.
\end{theorem} 
\begin{proof}
	First recall that by Theorem~\ref{thm:diameter} we have $d=3$, and so $\gamma=4$. We will first show that $\G$ is edge-regular with $\lambda=2$. Pick an arbitrary edge $xy$ and observe that by Proposition~\ref{k=5:prop1} we obtain $|D^1_2(x,y)|=2$, which forces $|D^1_1(x,y)|=2$. This shows that $\G$  is edge-regular with $\lambda=2$. It follows that for every vertex $x$ of $\G$, the subgraph of $\G$ which is induced on $\G(x)$, is isomorphic to the five-cycle $C_5$. By \cite[Proposition~1.1.4]{BCN}, $\G$ is isomorphic to the icosahedron. 
\end{proof}

\medskip \noindent
{\bfseries{\scshape{Proof of Theorem~\ref{thm:main}}.}} It is straightforward  to see that all graphs from Theorem~\ref{thm:main} are regular NDB graphs with $\gamma=d+1$. Assume now that $\G$ is a regular NDB graph with valency $k$, diameter $d$ and $\gamma=d+1$. If $d=2$, then it follows from Remark \ref{rem:d=2} that $\G$ is isomorphic either to the Petersen graph,
the complement of the Petersen graph, the complete multipartite graph $K_{t \times 3}$ with $t$ parts of cardinality $3$ ($t \ge 2$), the M\"obius ladder graph on eight vertices, or the Paley graph on 9 vertices. If $d \ge 3$, then it follows from Theorem~\ref{thm:diameter} that $k \in \{3,4,5\}$. If $k=3$, then $\G$ is isomorphic to the $3$-dimensional hypercube $Q_3$ by Theorem~\ref{thm:k3d3}. If $k=4$ then $\G$ is isomorphic to the line graph of $Q_3$ by Theorem~\ref{k44mm}. If $k=5$, then $\G$ is isomorphic to the icosahedron by Theorem~\ref{k=5:thm1}. \hfill $\Qed$ \\

\section{Acknowledgements}
This work is supported in part by the Slovenian Research Agency (research program P1-0285, research projects N1-0062, J1-9110, J1-1695, N1-0140, N1-0159, J1-2451, N1-0208 and Young Researchers Grant).


\end{document}